\documentclass[letterpaper, 11pt]{amsart}

\usepackage{fullpage} 
\usepackage{graphicx} 
\usepackage{amsmath, amssymb, amsfonts, amsthm}
\usepackage{mathtools, mathrsfs} 
\usepackage{bm} 
\usepackage{enumerate} 
\usepackage{caption} 
\usepackage[caption=false]{subfig} 
\usepackage{multirow} 
\usepackage{algorithmic} 
\usepackage{algorithm} 

\usepackage{tikz} 
\usetikzlibrary{arrows.meta} 
\usepackage{color}
\definecolor{darkblue} {rgb} {0, 0, 0.545} 
\definecolor{darkgreen} {rgb} {0, 0.392, 0} 
\usepackage{url} 
\usepackage[plainpages = false, colorlinks, citecolor = blue, filecolor = black, linkcolor = red, urlcolor = darkblue]{hyperref} 
\usepackage[capitalize]{cleveref}
\crefname{subsection}{Subsection}{Subsections}
\crefname{equation}{Equation}{Equations}
\crefname{figure}{Figure}{Figures}

\theoremstyle{definition} 
\newtheorem{definition}{Definition}[section] 
\newtheorem{remark}[definition]{Remark}
\theoremstyle{plain} 
\newtheorem{theorem}[definition]{Theorem}

\newtheorem{lemma}[definition]{Lemma}

\allowdisplaybreaks 

\newcommand{\R}{\mathbb{R}} 
 
\newcommand{\ie}{{\it{i.e.}, }}

\newcommand{\del}{\partial}

\newcommand{\tr}{\textnormal} 
\newcommand{\ds}{\displaystyle}

\newcommand{\x}{\mathbf{x}}

\newcommand{\D}{\mathcal{D}}  				
\newcommand{\B}{\mathcal{B}} 				
\newcommand{\F}{\mathcal{F}} 				
\newcommand{\bond}{\mathrm{Bo}} 				
\newcommand{\n}{\mathbf{\hat{n}}}  				
\newcommand{\avgop}{M}                          		
\newcommand{\bxi}{\bm{\xi}} 					

\usepackage[backend = bibtex, giveninits = true, style = alphabetic, natbib = true, url = false, sorting = nyt, doi = true, abbreviate = true, backref = false]{biblatex}
\renewbibmacro{in:}{\ifentrytype{article}{}{\printtext{\bibstring{in}\intitlepunct}}}

\begin{document}
\title{High Spots for the Ice-Fishing \\ Problem with Surface Tension} 
\thanks{C.H. Tan and B. Osting acknowledge partial support from NSF DMS 17-52202.} 
\author{Nathan Willis} 
\address{Department of Mathematics, University of Utah, Salt Lake City, UT} 
\email{willis@math.utah.edu} 
\author{Chee Han Tan}
\address{Department of Mathematics, Bucknell University, Lewisburg, PA} 
\email{cht007@bucknell.edu} 
\author{Christel Hohenegger} 
\address{Department of Mathematics, University of Utah, Salt Lake City, UT} 
\email{choheneg@math.utah.edu} 
\author{Braxton Osting} 
\address{Department of Mathematics, University of Utah, Salt Lake City, UT} 
\email{osting@math.utah.edu} 


\date{\today}


\begin{abstract} 
In the ice-fishing problem, a half-space of fluid lies below an infinite rigid plate (``the ice'') with a hole.  In this paper, we investigate the ice-fishing problem including the effects of surface tension on the free surface. The dimensionless number that describes the effect of surface tension is called the Bond number.  For holes that are infinite parallel strips or circular holes, we transform the problem to an equivalent eigenvalue integro-differential equation on an interval and expand in the appropriate basis (Legendre and radial polynomials, respectively). We use computational methods to demonstrate that the high spot, \ie the maximal elevation of the fundamental sloshing profile, for the IFP is in the interior of the free surface for large Bond numbers, but for sufficiently small Bond number the high spot is on the boundary of the free surface. While several papers have proven high spot results in the absence of surface tension as it depends on the shape of the container, as far as we are aware, this is the first study investigating the effects of surface tension on the location of the high spot. \end{abstract} 

\maketitle


\section{Introduction} \label{sec:Intro}
Sloshing refers to the motion of a liquid free surface, \ie the interface between the liquid in the container and the air above, inside partially filled containers \cite{Ibrahim, Faltinsen}. Liquid sloshing is a ubiquitous phenomenon, ranging from the oscillation of fuel in road tank vehicles and liquid-propellant rockets to seiches in lakes and harbors induced by earthquakes to our everyday experience in carrying a cup of coffee. Liquid sloshing has detrimental impacts on the stability and structural safety of stationary or moving vessels. For example, violent fuel sloshing within spacecraft fuel tanks produces highly localized pressure on tank walls, leading to deviation from its planned flight path or compromising its structural integrity.

Surface tension, defined as a force per unit length, is the intermolecular force required to contract the liquid surface to its minimal surface area. Examples of surface tension effects include the nearly spherical shape of liquid droplets and the ability of small insects to walk on water. The dimensionless parameter measuring the relative magnitudes of gravitational and surface tension forces is referred to as the Bond number and given by $\bond = \rho g\ell_c^2/\sigma$, where $\rho > 0$ is the constant fluid density, $\ell_c$ is a characteristic length scale of the container, and $\sigma > 0$ is the surface tension coefficient. For example, in a microgravity environment, the magnitude of body forces is tiny and surface tension forces predominate. Mathematically, surface tension is incorporated into the sloshing model via the Young-Laplace equation, $ \Delta P = -2\sigma H$, which asserts that the pressure difference $\Delta P$ between the inside and the outside of the fluid free surface is proportional to the mean surface curvature $H$. Recently a variational characterization of fluid sloshing with surface tension was derived in \cite{tan:2017} and an isoperimetric problem was considered in \cite{Tan2021}. 

In this paper, we investigate the ice-fishing problem (IFP), including the effects of surface tension on the free surface. The  IFP studies the problem of free oscillations of an incompressible, inviscid fluid for an irrotational flow in a half space bounded above by an infinite rigid plane where the free surface is some aperture in the plane; see \cref{fig:IFP_domain_mono}. We consider the cases where the aperture is either a circular hole or an infinite parallel strip. We denote the equilibrium free surface by $\F$, the wetted boundary by $\B$, and the fluid domain by $\D$. 

\begin{figure}[tbhp] 
\begin{center} 
\begin{tikzpicture} [scale =0.7, every node/.style={scale=0.7}]
    
	\def \r{5};

	\begin{scope}
		\clip (-\r,0) rectangle (\r,-\r);
		\filldraw[color = white, even odd rule,inner color=blue,outer color=white] (0,0) circle (\r);
	\end{scope}

	\def \s{0}; 
	\draw [darkblue, line width = 3.5pt,stealth-] (-\r - \s, 0) - - (-1 - \s, 0); 
	\draw [darkblue, line width = 3.5pt,-stealth] (1 - \s, 0) - - (\r - \s, 0) node [right] {\large $x$}; 
	\draw [ultra thick, -Stealth] (0 - \s, 0) - - (0 - \s, 1.7) node [above] {\large $z$}; 
	\draw (-1 - \s, 0.2) node [above] {\large $-1$}; 
	\draw (1 - \s, 0.2) node [above] {\large $1$}; 
	\draw (0 - \s, -3.5) node {\large $\mathcal{D}$}; 
	\draw [darkblue] (2.5 - \s, 0.5) node [] {\large $\mathcal{B}$}; 
	\draw (0.1, 0.3) node [ right] {\large $\mathcal{F}$}; 
	
	\draw [red, ultra thick] (-1, 0) arc(130:410:1.55cm); 
	\draw [green!80!black, ultra thick] plot[smooth, domain = -1:1] (\x, {1*((\x)^4+2*(\x)^3-2*\x-1)}); 
	\draw [green!80!black] (0.4, -0.6) node {\large $\overline {\mathcal{D}}$}; 
	\draw [red] (-0.3, -1.75) node {\large $\widetilde{\mathcal{D}}$}; 
		\filldraw [black] (-1 - \s, 0) circle (5pt); 
	\filldraw [black] (1 - \s, 0) circle (5pt); 
\end{tikzpicture} 
\end{center} 
\caption[Fluid domain for a cross section of the infinite parallel strip ice-fishing problem]{An illustration of the fluid domain $\D$ and boundary $\B\cup\F\cup\del\F$ for a cross section of the infinite parallel strip. Multiple bounded domains are shown to demonstrate the domain monotonicity property of the fundamental sloshing frequency, \ie $\omega_1^2(\overline{\D})\le \omega_1^2(\widetilde{\D})$.} 
\label{fig:IFP_domain_mono}
\end{figure}
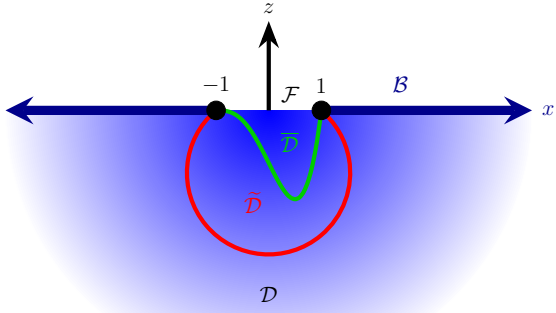


\subsection{Previous results} 
In the absence of surface tension, Moiseev \cite{moiseev:1964} established the property of domain monotonicity for the square of the fundamental (smallest) sloshing frequency $\omega_1^2$. Namely, for any two bounded containers $\overline{\D}, \widetilde{\D}$ with an identical $\F$ and $\overline{\D}\subset\widetilde{\D}$, we have that $\omega_1^2(\overline{\D})\le \omega_1^2(\widetilde{\D})$; see \cref{fig:IFP_domain_mono}. This result is an immediate consequence of the variational characterization of $\omega_1^2$.  It follows that the fundamental sloshing frequency for the IFP furnishes the universal upper bound for the fundamental sloshing frequency of arbitrary containers with coinciding $\F$. In the presence of surface tension, it was shown that this domain monotonicity result continues to hold for a free surface that is freely allowed to move at its boundary \cite{tan:2017}.

The IFP with zero surface tension has been well-studied in the past few decades. Davis \cite{davis:1970} reformulated the infinite parallel strip problem as an integral equation involving Green's function. He expressed the velocity potential as the infinite sum of Legendre polynomials and applied the principle of deformation of contours to give $\{\omega_j^2\}_{j = 1}^\infty$ as the eigenvalues of an infinite symmetric matrix. Furthermore, he obtained approximations for $\omega_j^2$ by computing the eigenvalues of the truncated matrix and derived a fourth-order asymptotic expansion for higher eigenvalues. Henrici, Troesch, and Wuytack \cite{henrici:1970} formulated the IFP with a circular or strip-like aperture including a decay condition at infinity for the fluid velocity field and derived an equivalent Fredholm integral equation for the velocity potential in the aperture using potential theory. Miles \cite{miles:1972} recasts the IFP as formulated by Henrici et al. \cite{henrici:1970} to a homogeneous Fredholm integral equation for the velocity distribution in the aperture. Troesch \cite{troesch:1973} transformed the IFP onto a bounded domain using the Kelvin inversion. Fox and Kuttler \cite{fox:1983} transformed the infinite parallel strip problem into an equivalent weighted problem on a semi-infinite strip employing a conformal map. Most authors computed upper bounds for $\omega_j^2$ on their equivalent problems using the Rayleigh-Ritz method.

Of particular interest is the problem of determining the location of the high spots, \ie the maximal elevation of the sloshing profile $\xi$. {\it Unless otherwise noted, when discussing the high spots it is assumed that $\xi$ is the fundamental sloshing profile.} The study of these high spots is motivated as a fluid-analogue to the Hot Spots Conjecture: ``For any second eigenfunction of the Neumann Laplacian, the extremal values of this eigenfunction are only attained on the boundary of the triangle'' \cite{polymath}. Indeed, \cite[Proposition 3.1]{kulczycki:2009} proved that the hot spots problem and the high spots problem are equivalent for an upright cylindrical tank in the absence of surface tension. In two dimensions, the high spot is located on $\del\F$ for any $\B$ that can be written as a negative $C^2$ function on $\F$ such that $\B$ is not tangent to $\F$ at their common endpoints \cite{kulczycki:2009}. This result extends to three dimensions when considering the previously described two-dimensional container as the cross section of a finite canal \cite{kulczycki:2011}. For a radially-symmetric, convex, bounded container such that $\D$ is contained in the upright cylinder $\F\times (-\infty, 0)$, the high spot is located on $\del\F$ \cite{kulczycki:2012}. Concerning the IFP, it is known that the high spot is located in the interior of $\F$ when $\F$ is either a circular hole or an infinite parallel strip \cite{kulczycki:2009}. All these results rely on the property that $\xi$ is proportional to the trace of the fundamental sloshing mode $\Phi_1$ on $\overline\F$ when the fluid oscillates freely with the fundamental sloshing frequency. Finally, it was conjectured that for a bounded planar domain with smooth $\B$ such that at least one angle between $\B$ and $\F$ is greater than $\pi/2$, the high spot is located in the interior of $\F$ \cite{kulczycki:2009}.


\subsection{Main results} \label{sec:HS1} 
In this paper, we investigate how the presence of surface tension affects the location of the high spot for IFP, focusing on an infinite parallel strip and a circular hole. We use computational tools to demonstrate that the high spot is in the interior of $\F$ for large $\bond$, but for sufficiently small $\bond$ the high spot is on $\del\F$. We plot the location of the high spot on the $x$-axis in the infinite parallel strip in \cref{fig:IFP_HighSpot_b} and the circular hole in \cref{fig:IFP_HighSpot_c} for varying $\bond$. The  fundamental sloshing profiles for $\bond=1,20,100$ are shown in \cref{fig:IFP_HighSpot_a}. In both cases, \emph{we observe that as $\bond$ increases, the high spot moves from the boundary of $\F$ to the interior of $\F$}. This transition happens at $\bond^*=8.98461$ for the infinite parallel strip and at $\bond^*=4.63462$ for the circular hole. The vertical asymptote corresponds to the high spot location for $\bond\to\infty$, \ie in the absence of surface tension. To obtain these results we reduce the three-dimensional problem to a one-dimensional integro-differential equation. This one-dimensional problem is then solved approximately in orthogonal polynomial spaces, specifically Legendre polynomials for the infinite parallel strip and radial polynomials for the circular hole, resulting in a generalized eigenvalue problem for sloshing frequencies. For a circular hole, the critical value $\bond^*$ is obtained by finding the value of $\bond$ at which the concavity of the sloshing profile at the boundary switches from positive to negative. We show that this condition can be reformulated as a fixed point problem, yielding a precise value of $\bond^*$. 

\begin{figure}[tbhp] 
\centering
	\subfloat[Infinite parallel strip]{\label{fig:IFP_HighSpot_b}
	\includegraphics[width = 0.32\textwidth,trim = {0.5cm 0.18cm 2cm 1.8cm},clip=true]{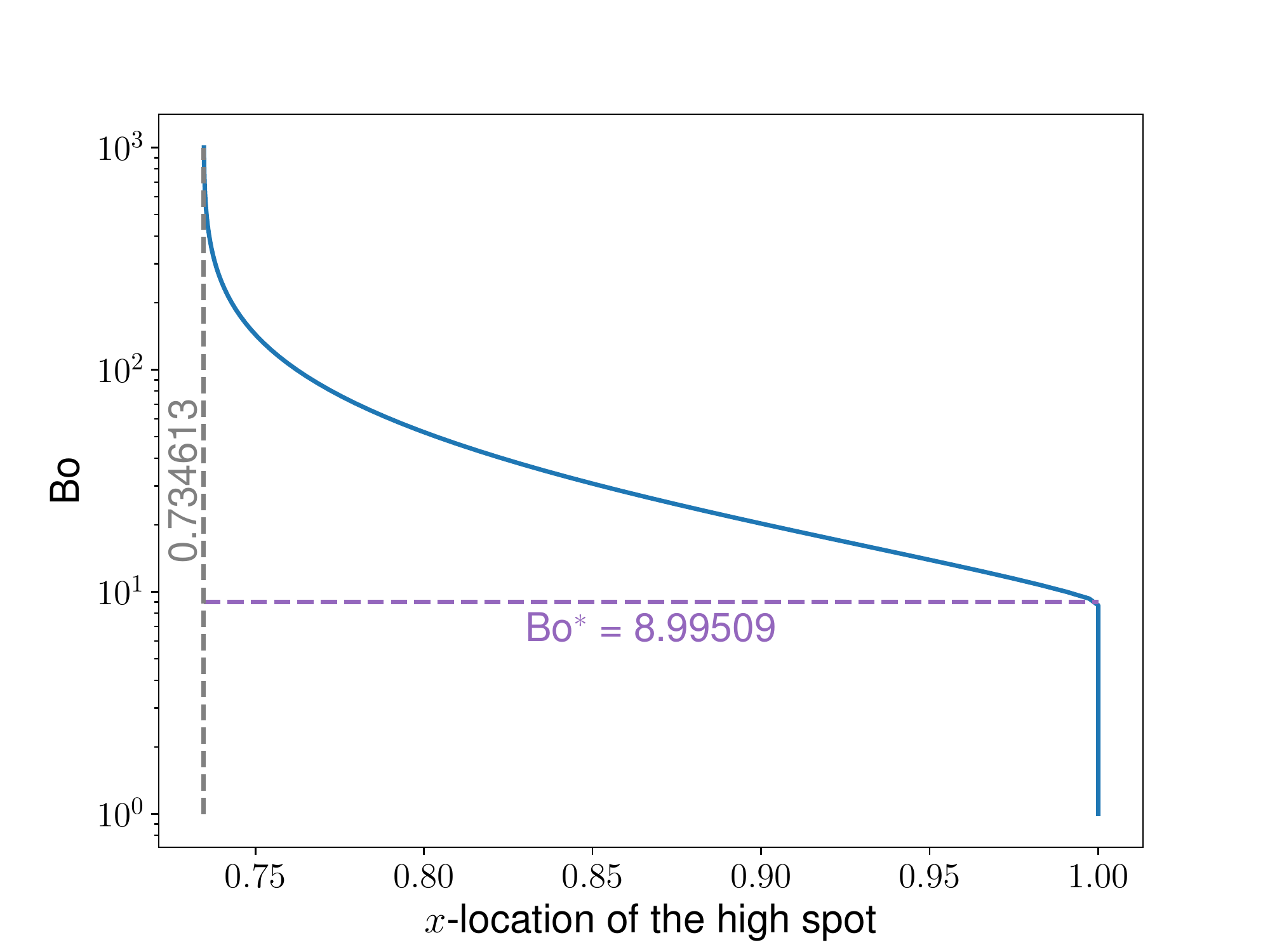}}
	\subfloat[Circular hole]{\label{fig:IFP_HighSpot_c}
	\includegraphics[width = 0.32\textwidth,trim = {0.5cm 0.18cm 2cm 1.8cm},clip=true]{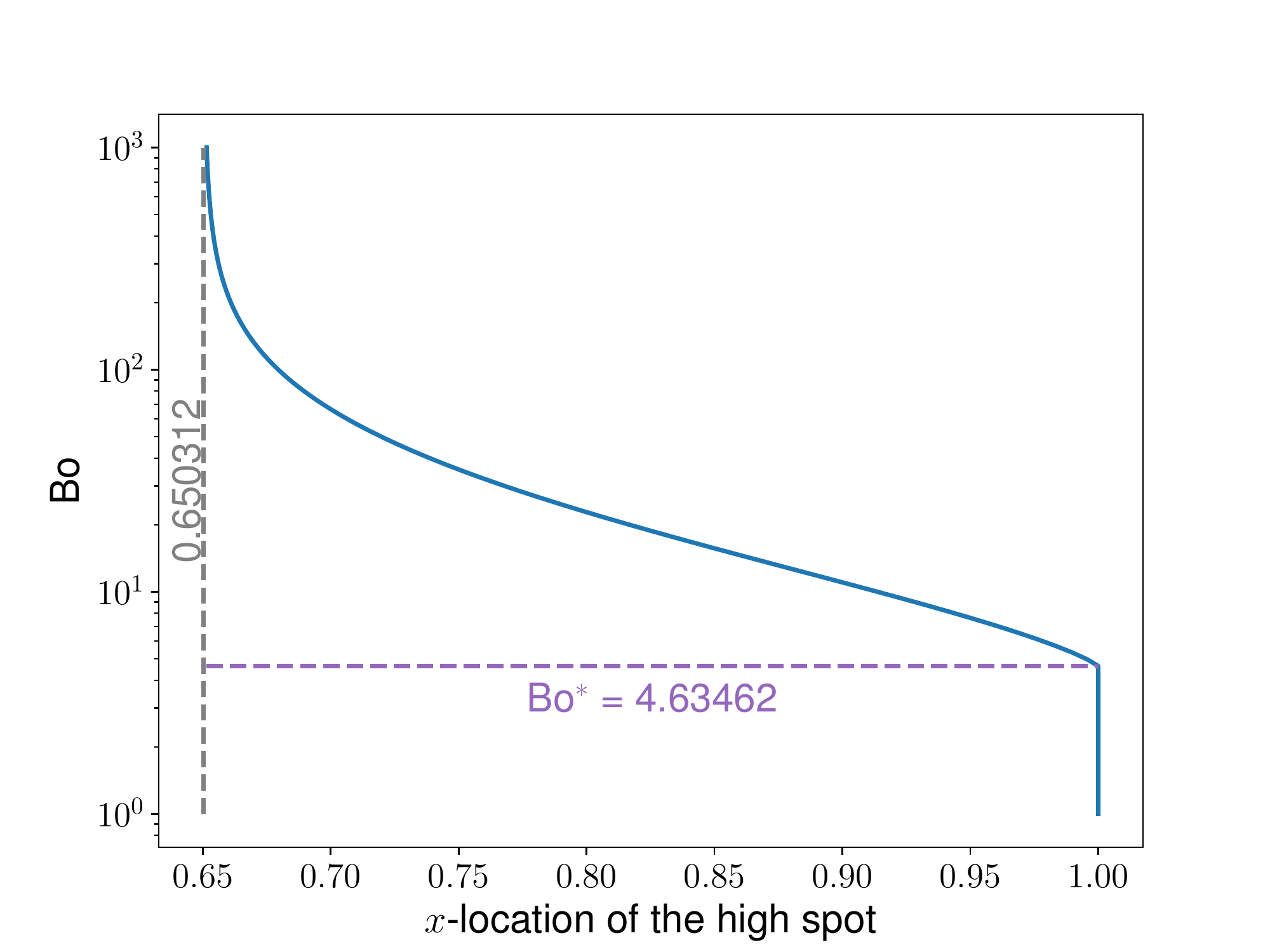}} 
	\subfloat[Fundamental sloshing profile]{\label{fig:IFP_HighSpot_a}\includegraphics[width = 0.32\textwidth,trim = {0.5cm 0.2cm 2cm 1.8cm},clip=true]{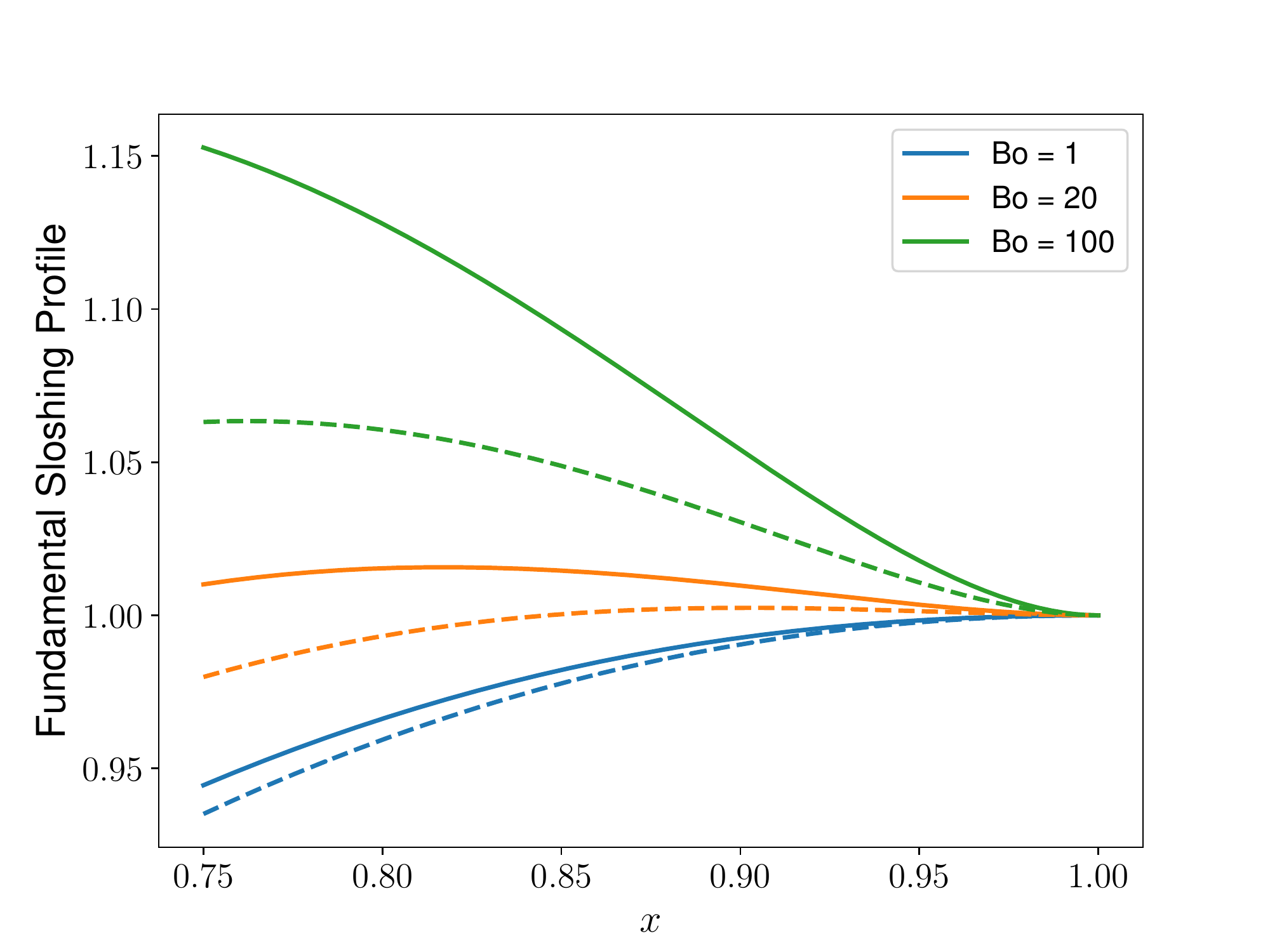}}
\caption{Plot of the location of the high spot for (a) an infinite parallel strip and (b) a circular hole for varying $\bond$. The location of the high spot for the limiting case $\bond \to \infty$ is marked with a vertical asymptote. (c) The fundamental sloshing profiles near the boundary $x = 1$ for both the infinite parallel strip (dashed lines) and the circular hole (solid lines) for $\bond = 1, 20, 100$.}
\label{fig:IFP_HighSpot}
\end{figure} 


\subsection{Outline} 
This paper is structured as follows. 
In \cref{s:modelDeriv}, we describe the derivation of IFP with surface tension and an equivalent Fredholm integro-differential equation. In \cref{sec:strip,sec:hole}, we describe the reduction of this equation when the aperture is an infinite parallel strip and a circular hole, respectively. In each section, we start by writing out the specific integro-differential equation, before deriving the associated weak form. The weak form is numerically solved in a Legendre polynomial space in \cref{sec:Legendre} and in a radial polynomial space in \cref{sec:Legendre3d}, yielding a generalized eigenvalue problem. Examples of sloshing profiles and frequencies as well as convergence of the numerical schemes are presented in \cref{sec:numerics,sec:numerics_hole}. In \cref{sec:HS}, we provide a justification for the location of the high spot for a radial hole that is based on determining the concavity at $r = 1$. This justification depends on the numerical results, but is intended to provide a basis for future analytical work on this point. We conclude in \cref{sec:IFP_disc} with a discussion.


\section{Model Derivation} \label{s:modelDeriv} 
The IFP can be described as the limiting case of the linear sloshing problem in a bounded domain with identical aperture; see \cite[Appendix A]{tan:2017} for a detailed derivation of the linear sloshing problem with surface tension on a bounded domain. We choose the halfwidth of the equilibrium free surface $\F$ as the characteristic length scale $\ell_c$ and nondimensionalize all lengths by $\ell_c$, time by $t_c\coloneqq\sqrt{\ell_c/g}$, and velocity by $\ell_c/t_c$. Let $\x = (x, y, z)$ be dimensionless Cartesian coordinates such that the $z$-axis is directed vertically upward, perpendicular to the ice sheet. Let $\Phi(\x)$ be the velocity potential with a time-harmonic factor $\cos(\omega t)$ removed and $\xi(x,y)$ be the sloshing profile, \ie the free surface displacement, with a time-harmonic factor $\sin(\omega t)$ removed, where $\omega$ is the natural sloshing frequency. Then, the IFP with Neumann boundary conditions is a dimensionless linear boundary spectral problem for $(\omega, \Phi, \xi)$ defined by: 
\begin{subequations} \label{eq:IFP} 
\begin{alignat}{3}  
\label{eq:IFP1} \Delta\Phi & = 0 && \ \ \tr{ in } \ \ && \D = \left\{\x\in\R^3\colon z < 0\right\}, \\ 
\label{eq:IFP2} \del_z\Phi & = 0 && \ \ \tr{ on } \ \ && \B , \\ 
\label{eq:IFP3} \del_z\Phi & = \omega\xi && \ \ \tr{ on } \ \ && \F, \\ 
\label{eq:IFP4} \xi - \frac{1}{\bond}\Delta_{\F}\xi & = \omega\Phi && \ \ \tr{ on } \ \ && \F, \\ 
\label{eq:IFP5} \partial_{\n_\F}\xi & = 0 && \ \ \tr{ on } && \del\F ,\\
\label{eq:IFP6}\int_\F \xi\, dA & = 0. 
\end{alignat} 
\end{subequations} 
In the above equations, $\Delta_{\F}$ is the Laplacian operator on the free surface, and $\partial_{\n_\F}$ is the derivative in the direction normal to the boundary of the free surface in the plane $z = 0$. The two cases we consider are
\begin{equation*}
\begin{array}{r c l | r c l}
\multicolumn{3}{c|}{\text{infinite parallel strip}}  & \multicolumn{3}{c}{\text{circular hole}} \\ \hline
    \F&=&\left\{\x\in\R^3\colon |x|<1, z=0\right\}  &     \F&=&\left\{\x\in\R^3\colon |(x,y)|<1, z=0\right\} \\
    \B&=&\left\{\x\in\R^3\colon |x|>1,z=0\right\}   &     \B&=&\left\{\x\in\R^3\colon |(x,y)|>1,z=0\right\} \\
\del\F&=&\{\x\in\R^3\colon |x| = 1, z = 0\}         & \del\F&=&\{\x\in\R^3\colon |(x,y)|=1,z=0\}
\end{array}
\end{equation*} 
The boundary condition \cref{eq:IFP5} is necessary as surface tension introduces the linearized curvature term $\Delta_\F\xi$ into the model. It is known as the \emph{contact line boundary condition} and it prescribes how the fluid free surface moves along the container wall. In this paper we consider the \emph{free-end edge constraint} $\del_{\n_\F}\xi = 0$ on $\del\F$. 

\begin{remark} 
For $(\omega, \Phi, \xi)$ satisfying \cref{eq:IFP}, we have that \cref{eq:IFP6} additionally gives no flow at infinity, \ie $\left|\nabla \Phi\right| \to 0 \tr{ as } |\x| \to \infty$. This result was shown in \cite{henrici:1970} by expanding the kernel of the solution operator in spherical harmonics. 
\end{remark} 

We solve \cref{eq:IFP} by first expressing $\Phi$ as an integral operator acting on the Neumann boundary data $\omega\xi$, satisfying \cref{eq:IFP1}-\cref{eq:IFP3}. We then plug this expression for $\Phi$ into \cref{eq:IFP4} to derive an equivalent eigenvalue integro-differential equation for $(\omega, \xi)$. We start by considering the Laplace problem with compactly supported Neumann data
\begin{subequations} \label{eq:Lap} 
\begin{alignat}{3}  
\label{eq:Lap1} \Delta\Phi(\x) & = 0 && \ \ \tr{ in } \ \ && \D = \left\{\x\in\R^3\colon z < 0\right\}, \\ 
\label{eq:Lap2} \del_z\Phi(\x) & = f(\x)\chi_{\F}(\x)&& \ \ \tr{ on } \ \ && \del\D = \left\{\x\in\R^3\colon z = 0\right\}.
\end{alignat} 
\end{subequations}
 
 \begin{remark} \label{rem:Uni}
By defining $f = \omega\xi$ it is easy to see that \cref{eq:IFP5} is the compatibility condition, $\int_{\del\D} f\, d\x = 0$, guaranteeing the existence and uniqueness, up to a constant, of the Neumann problem for $\Phi$ in the half-space.
\end{remark}

\begin{remark} \label{rem:uni_xsi}
If $(\omega, \Phi, \xi)$ is a solution of \cref{eq:IFP}, then so are $(\omega, -\Phi, -\xi), (-\omega, -\Phi, \xi)$, and $(-\omega, \Phi, -\xi)$. There is no contradiction with the previous remark as for a given $(\omega, \xi)$ pair, $\Phi$ is unique up to a constant.
\end{remark} 

\begin{definition} \label{def:S}
Let $S$ be the operator, such that $\Phi(\x)$ satisfying \cref{eq:Lap} can be written as $\Phi(\x) = S[f\chi_{\F}](\x) +\Phi_\infty$, where $\Phi_\infty$ is constant and $Sf(\x)\to 0$ as $|\x|\to\infty$.
\end{definition}
Examples of $S$ for  the infinite strip and circular holes are given in \cref{sec:strip,sec:hole}, respectively. Comparing \cref{eq:Lap2} and \cref{eq:IFP2}-\cref{eq:IFP3}, we set $f = \omega\xi$, and plug in \cref{eq:IFP4}. This gives the following equation on the free surface, 
\begin{equation} \label{eq:IDEmid}
\xi - \frac{1}{\bond}\Delta_{\F} \xi = \omega^2 \widehat{S}\xi + \omega\Phi_\infty. 
\end{equation}
Here $\widehat{S}$ is the restriction of $S$ to the free surface $z = 0$. To remove $\Phi$ from this equation entirely, resulting in an eigenvalue integro-differential equation for $(\omega^2, \xi)$, we utilize \cref{eq:IFP6}. Motivated by \cite{henrici:1970}, we define the mean value operator to be $\avgop f= \frac{1}{|\F|} \int_{\F} f\, dA$ and next we consider $(I-\avgop)$ where $I$ is the identity operator. Note that $(I-\avgop)$ applied to a constant yields 0 and $(I-\avgop)\xi = \xi$. Therefore, applying $(I-\avgop)$ to \cref{eq:IDEmid} and including the appropriate conditions on $\xi$ yields
\begin{subequations} \label{eq:IDEmain}
    \begin{alignat}{3}
        \label{eq:IDEmain1} \xi - \frac{1}{\bond}(I-\avgop)\Delta_{\F} \xi &= \omega^2 (I-\avgop)\widehat{S}\xi && \ \ \tr{ on } \ \ && \F, \\ 
        \label{eq:IDEmain2} \partial_{\n_\F}\xi & = 0 && \ \ \tr{ on } && \del\F ,
    \end{alignat}
\end{subequations}
together with $\int_\F \xi\, dA  = 0.$
The following theorem establishes equivalence between solutions of  \cref{eq:IFP} and \cref{eq:IDEmain}.
\begin{theorem} \label{thm:Equiv}
    Let $S$ be as in \cref{def:S} with $f = \omega\xi$. If $(\omega, \xi, \Phi)$ is a solution to \cref{eq:IFP} then $(\omega^2, \xi)$ solves \cref{eq:IDEmain}. Moreover, if $(\omega^2, \xi)$ is a solution to \cref{eq:IDEmain} and we define $\Phi = \omega S \xi - \omega \avgop \widehat{S}\xi - \frac{1}{\omega \bond}\avgop \Delta \xi,$ then $(\omega, \xi, \Phi)$ is a solution to \cref{eq:IFP}.
\end{theorem} 

\begin{proof} 
The first statement is a direct result of the derivation of \cref{eq:IDEmain} along with \cref{rem:Uni}. The second statement follows from the facts that $\omega \avgop \widehat{S}\xi$ and $\frac{1}{\omega \bond}\avgop \Delta \xi$ are constant and therefore \cref{eq:IFP1}-\cref{eq:IFP3} are trivially satisfied due to the operator $S$. Restricting the definition of $\Phi$ to the free surface, we have $\Phi = \omega \left( I -\avgop \right) \widehat{S}\xi - \frac{1}{\omega \bond}\avgop \Delta \xi$. Therefore, from \cref{eq:IDEmain1} we find
    \begin{equation*}
        \omega \Phi = \xi - \frac{1}{\bond}(I-\avgop)\Delta \xi - \frac{1}{\bond}\avgop \Delta \xi
    \end{equation*}
    such that $(\omega, \xi,\Phi)$ satisfies \cref{eq:IFP4}.
\end{proof}


\section{Infinite parallel strip}\label{sec:strip}
In the case of the infinite parallel strip, we assume no dependence on $y$ and therefore consider a cross section as in \cref{fig:IFP_domain_mono}. Using \cite{henrici:1970,wang:2014}, we have the following representation of bounded solutions. 

\begin{lemma} \label{lem:IDE_strip}
In the case of an infinite parallel strip with $\Phi(\x) = \Phi(x, z)$, there exists a bounded solution to \cref{eq:Lap} of the form $\Phi(x, z) = S[f\chi_{[-1,1]}](x, z) + \Phi_\infty$, where
\begin{equation*}
S[f\chi_{[-1,1]}](x, z) = -\frac{1}{2\pi}\int_{-1}^1 \ln\left((x - s)^2 + z^2\right)f(s)\, ds, \ \ \Phi_\infty = \lim_{|(x,z)|\to\infty} \Phi(x, z).  
\end{equation*}  
\end{lemma} 
\begin{proof}
The existence up to a constant follows from the compatibility condition. Since $(Sf)(x, z)\to 0$ as $\left|(x ,z)\right|\to\infty$ \cite{vannekoski:2014}, the constant is $\Phi_\infty=\lim_{|(x,z)|\to\infty} \Phi(x,z)$. The facts that $\Phi$ satisfies \cref{eq:Lap} and that $\Phi$ is bounded can be found in \cite{wang:2014}. 
\end{proof}
\cref{lem:IDE_strip} gives the form of $S$ as in \cref{def:S} with restriction to $z=0$ given by $\widehat S\xi(x) = -\frac{1}{\pi}\int_{-1}^1 \ln|x - s|\xi(s)\, ds.$ Therefore, using \cref{thm:Equiv}, plugging into \cref{eq:IDEmain} and keeping in mind that $\Phi(\x)=\Phi(x,z)$, we are seeking solutions of the following integro-differential equation
\begin{subequations} \label{eq:IFP_integral} 
\begin{alignat}{1} 
\label{eq:IFP_integral1} \xi - \frac{1}{\bond}(I - M)\xi'' & = \omega^2(I - M)\widehat S\xi = \lambda\widehat S\xi \ \ \tr{ on $(-1, 1)$}, \\ 
\label{eq:IFP_integral2} \xi'(\pm 1) & = 0, \qquad \int_{-1}^1 \xi\, dx = 0. 
\end{alignat} 
\end{subequations} 

We now turn our attention to seeking weak eigenpairs of \cref{eq:IFP_integral}. Define the Hilbert space $V = \left\{\xi\in H^2(-1, 1)\colon \xi'(\pm 1) = 0, \, \int_{-1}^1 \xi\, dx = 0\right\}$. 

\begin{definition} \label{def:weak} 
We say that $(\omega, \xi)\in\R\times V$ is a weak sloshing eigenpair of \cref{eq:IFP_integral} if the following holds for all $g\in V$: 
\begin{equation} \label{eq:IFP_weak}
\int_{-1}^1 \left(\xi g + \frac{1}{\bond}\xi'g'\right)\, dx = -\frac{\omega^2}{\pi}\int_{-1}^1\int_{-1}^1 \ln|x - s|\xi(s)g(x)\, ds dx. 
\end{equation} 
\end{definition} 
\cref{eq:IFP_weak} is formally obtained by multiplying \cref{eq:IFP_integral1} by $g$, integrating from $-1$ to $1$, using integration by parts with the prescribed boundary conditions, and simplifying with the facts that $M\xi''$ and $M\widehat{S}\xi$ are constant and $\int_{-1}^1 g dx = 0$.


\subsection{Polynomial Approximation} \label{sec:Legendre} 
We seek approximate solutions to \cref{eq:IFP_weak} in a finite-dimensional polynomial space. Let $p_j$ denote the normalized Legendre polynomial of degree $j\ge 1$ on $[-1, 1]$ with respect to the weight $1$, \ie $\int_{-1}^1 p_jp_k\, dx = \delta_{jk}$. Following \cite{shen:1994}, we define the spaces $P^{(n)} = \mathrm{span}\left\{p_0(x), p_1(x), \dots, p_n(x)\right\}$ and $W^{(n)} = \left\{v\in P^{(n)}\colon v'(\pm 1) = 0, \, \int_{-1}^1 v\, dx = 0\right\}.$ Suppose $(\omega, \xi)$ is a weak sloshing eigenpair of the IFP, as in \cref{def:weak}, and let $\xi^{(n)}$ be the orthogonal projection of $\xi$ to $W^{(n)}$ with respect to the $L^2(-1,1)$ inner product. In general \cite[Equation 9.4.10]{Canuto} we have 
\begin{equation*}
    \left\|\xi - \xi^{(n)}\right\|_{H^\ell (-1, 1)} \leq C n^{-1/2}n^{2\ell-m}\|\xi\|_{H^m(-1, 1)}
\end{equation*} 
for $\xi \in H^m(-1,1)$ with $1 \leq \ell \leq m$. In particular, since $\xi\in V \subset H^2(-1, 1)$ we have the a-priori convergence result
that $\xi^{(n)}\to\xi$ in $H^1(-1, 1)$ as $n^{-1/2}$ as $n\to\infty$. Furthermore, the convergence will be faster if $\xi$ is more regular and, in practice, we observe a higher rate of convergence; see \cref{sec:numerics}. 

For every $j = 1, 2, 3, \dots$, define the polynomial $q_j = (p_j - \beta_j p_{j+ 2})/\alpha_j$. First, we choose $\beta_j = \dfrac{j(j + 1)\sqrt{2j + 1}}{(j + 2)(j + 3)\sqrt{2j + 5}}$ to satisfy the Neumann boundary condition. Note that if Dirichlet boundary conditions were prescribed, one would simply define $\beta_j = \dfrac{\sqrt{2j + 1}}{\sqrt{2j + 5}}$, \ie all the boundary data is stored in the constants $\beta_j$. Next, we choose $\alpha_j = \sqrt{1 + \beta_j^2}$ to normalize the $q_j$ polynomials such that $\int_{-1}^1 q_j^2\, dx = 1$. 
\begin{lemma}
The set of polynomials $\{q_j\}_{j=1}^{n-2}$ for $n=3,4,\ldots$ constitutes a basis for $W^{(n)}$.
\end{lemma}
\begin{proof}
The statement follows from \cite{shen:1994,auteri:2003}.
\end{proof}

Therefore, the discrete weak formulation of the IFP is to find $(\omega, \xi^{(n)}) \in \R \times W^{(n)}$ such that for all $g\in W^{(n)}$: 
\begin{equation} \label{eq:IFP_discrete}
\int_{-1}^1 \left(\xi^{(n)} g + \frac{1}{\bond} \left(\xi^{(n)}\right)' g'\right)\, dx = -\frac{\omega^2}{\pi}\int_{-1}^1\int_{-1}^1 \ln|x - s|\xi^{(n)}(s)g(x)\, dsdx.
\end{equation} 
\begin{theorem}\label{thm:discrete_strip}
If $(\omega,\xi^{(n)})\in\R\times W^{(n)}$ solves the discrete IFP \cref{eq:IFP_discrete}, then $(\omega,\xi^{(n)})$ solves the generalized eigenvalue problem for $i=1,2,\ldots,n-2$
\begin{equation} \label{eq:IDE_numerics}  
\sum_{j = 1}^{n - 2} c_j \left(M_{ij} + \frac{1}{\bond}K_{ij}\right) = \omega^2 \sum_{j = 1}^{n - 2} c_j L_{ij}.
\end{equation} 
Here $M_{ij} = \int_{-1}^1 q_i(x)q_j(x)\, dx$ is the mass matrix, $K_{ij} = \int_{-1}^1 q_i'(x)q_j'(x)\, dx$ is the stiffness matrix, and $L_{ij} = -\int_{-1}^1\int_{-1}^1 \ln|x - s|q_i(x)q_j(s)\, dsdx$, which are given by
\begin{align*} 
\small
M_{ij} &= \delta_{ij} -\frac{\beta_{i}}{\alpha_{i}\alpha_{i+2}}\delta_{i+2,j} -\frac{\beta_{i-2}}{\alpha_{i-2}\alpha_{i}} \delta_{i-2,j}, \\
K_{ij} &= \frac{\beta_j \sqrt{2j + 5}}{\alpha_j^2 \sqrt{2j+1} }(2j + 1)(2j + 3) \delta_{ij}, \\
L_{ij} &= \frac{-\tilde{L}_{ij} + \beta_i \tilde{L}_{i+2,j} + \beta_j \tilde{L}_{i,j+2} - \beta_i \beta_j \tilde{L}_{i+2, j+2}}{\alpha_i \alpha_j}, \\
\tilde{L}_{ij} &= \begin{dcases}
\, \frac{4}{\pi}\frac{\sqrt{2i + 1}\sqrt{2j + 1}}{(i + j)(i + j + 2)(1 - (i - j)^2)} & \ \  \tr{ if } i+j \text{ is even, } \\
\, 0 & \ \ \tr{ otherwise}. 
\end{dcases}   
\end{align*}
\end{theorem}
\begin{proof}
We write $\xi^{(n)} = \sum_{j = 1}^{n - 2} c_j q_j$ and choose $g = q_i$ in \cref{eq:IFP_discrete}. Expanding the $q_i$ and $q_j$ polynomials and using \cite{davis:1970} we find the expression for $L_{ij}$. To compute $M_{ij}$ and $K_{ij}$, it suffices to consider the case where $i\le j$ since they are symmetric expressions. Using orthonormality of the Legendre polynomials, $p_i$, the expression for $M_{ij}$ follows. To obtain $K_{ij}$, we integrate by parts and use the fact that $q_i'(\pm 1) = 0$ for $W^{(n)}$ to get
\begin{equation*}
\small
    K_{ij} = -\frac{1}{\alpha_j}\left[\int_{-1}^1 q_i''p_j\, dx - \beta_m\int_{-1}^1 q_i''p_{j + 2}\, dx\right].
\end{equation*}
Since $q_i''$ is at most degree $i$ and the Legendre polynomial $p_i$ is orthogonal to any polynomial of lower degree, we see that $K_{ij} = 0$ if $i < j$ and 
\begin{equation*} 
K_{jj} = \frac{\beta_j}{\alpha_j^2}\int_{-1}^1 p_{j + 2}''p_j\, dx = \frac{\beta_j \sqrt{2j + 5}}{\alpha_j^2 \sqrt{2j+1} }(2j + 1)(2j + 3).
\end{equation*} 
The final value for $K_{jj}$ follows from  the orthonormality of $p_j$ the polynomials togehter with the well-known relationship,  
\begin{equation*} 
p_j''(x) = \sqrt{2j+1}\sum_{\substack{k=0 \\ k+j \text{ even}}}^{j-2} \left(k +\frac{1}{2}\right) \Big( j(j+1)-k(k+1) \Big) \frac{p_k(x)}{\sqrt{2k+1}}.
\end{equation*} 
\end{proof}


\subsection{Numerical Results} \label{sec:numerics} 
We now solve \cref{eq:IDE_numerics}, which is a generalized eigenvalue problem of the form 
\begin{equation} \label{eq:IFP_GEP} 
\left(\mathbf{M} + \frac{1}{\bond}\mathbf{K}\right) \bxi = \lambda\mathbf{L}\bxi, \ \ \lambda\coloneqq\omega^2, 
\end{equation}  
where $\left(\bm{\xi}\right)_i = c_i$ and $\mathbf{M}$, $\mathbf{K}$, $\mathbf{L}$, are square matrices. The mass matrix $\mathbf{M}$ is pentadiagonal, the stiffness matrix $\mathbf{K}$ is diagonal, and the matrix $\mathbf{L}$ is dense, but exhibits a checkerboard pattern as $L_{ij} = 0$ if $i+j$ is odd. Specifically, the only nonzero diagonals for $\mathbf{M}$ are the main diagonal and the second sub and super diagonals. It follows from \cref{thm:Equiv} that the eigenvalues $\lambda_j$ of \cref{eq:IFP_GEP} approximate the first $n$ of the natural sloshing frequencies of IFP squared. 

\begin{table}[tbhp] 
\begin{center}
    \begin{tabular}{|c | c | c | c || c |}\hline
              & $\bond = 1$  & $\bond = 10$ & $\bond = 50$ & $\bond = \infty$ \\ \hline
        $j=1$ & $7.0326$     & $2.5238$     & $2.1162$     & $2.0061$         \\ \hline
        $j=2$ & $37.7871$    & $6.9032$     & $4.1530$     & $3.4533$         \\ \hline
        $j=3$ & $119.0961$   & $16.5480$    & $7.4307$     & $5.1253$         \\ \hline
    \end{tabular} 
\end{center}
\caption{First three eigenvalues for the Infinite Parallel Strip IFP with  $\bond=\{1,10,50,\infty\}$.} 
\label{table:IFP_eigs_Neumann} 
\end{table} 

The first three eigenvalues of \cref{eq:IFP_GEP} for $\bond = 1, 10, 50,\infty$ with $n = 200$ are shown in \cref{table:IFP_eigs_Neumann}. 
To validate our solution to the infinite parallel strip IFP for $\bond = \infty$ we compared the eigenvalues to those found in Fox et al. \cite{fox:1983} and found that for all eigenvalues reported our results fit within the bounds provided. A table of all previous results \cite{henrici:1970,davis:1970,miles:1972} is given in \cite{fox:1983} and it is the case that Fox et al. provides the tightest bounds among the zero surface tension results. To get an understanding of the surface tension effects on the fundamental sloshing frequency we note that $\ds \frac{\lambda_1(\bond = 1)}{\lambda_1(\bond = \infty)} = 3.506$. That is, when the force due to surface tension is comparable to the gravitational force the fundamental sloshing frequency is increased to more than 300\% of the value when surface tension is negligible. 

\begin{figure}[tbhp]  
	\begin{center}
			\subfloat[$\lambda_1,\lambda_2$ convergence]{\label{fig:IFP_convergence_a}\includegraphics[width = 0.49\textwidth]{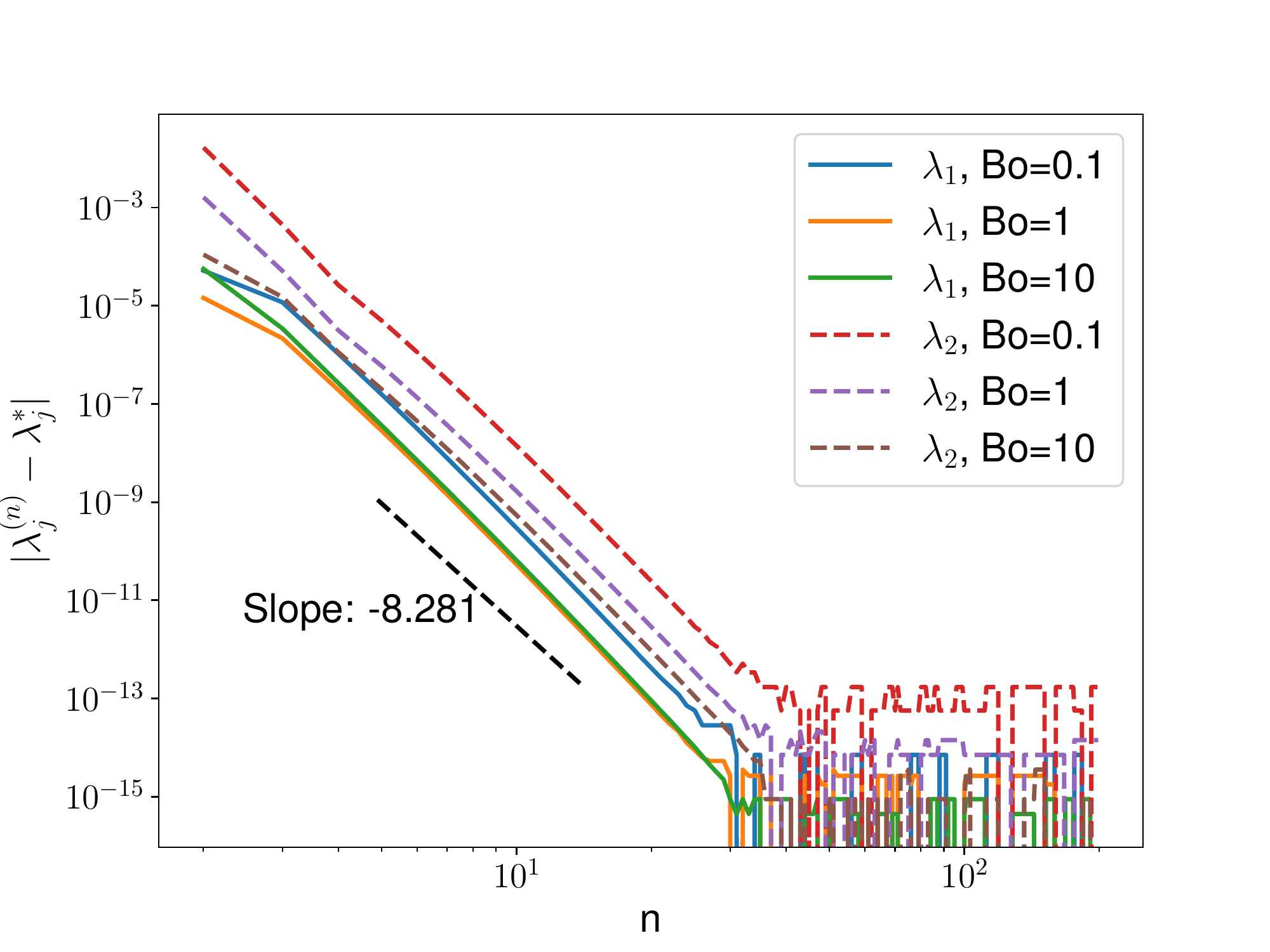}}  
			\subfloat[$\xi_1,\xi_2$ convergence]{\label{fig:IFP_convergence_b}\includegraphics[width = 0.49\textwidth]{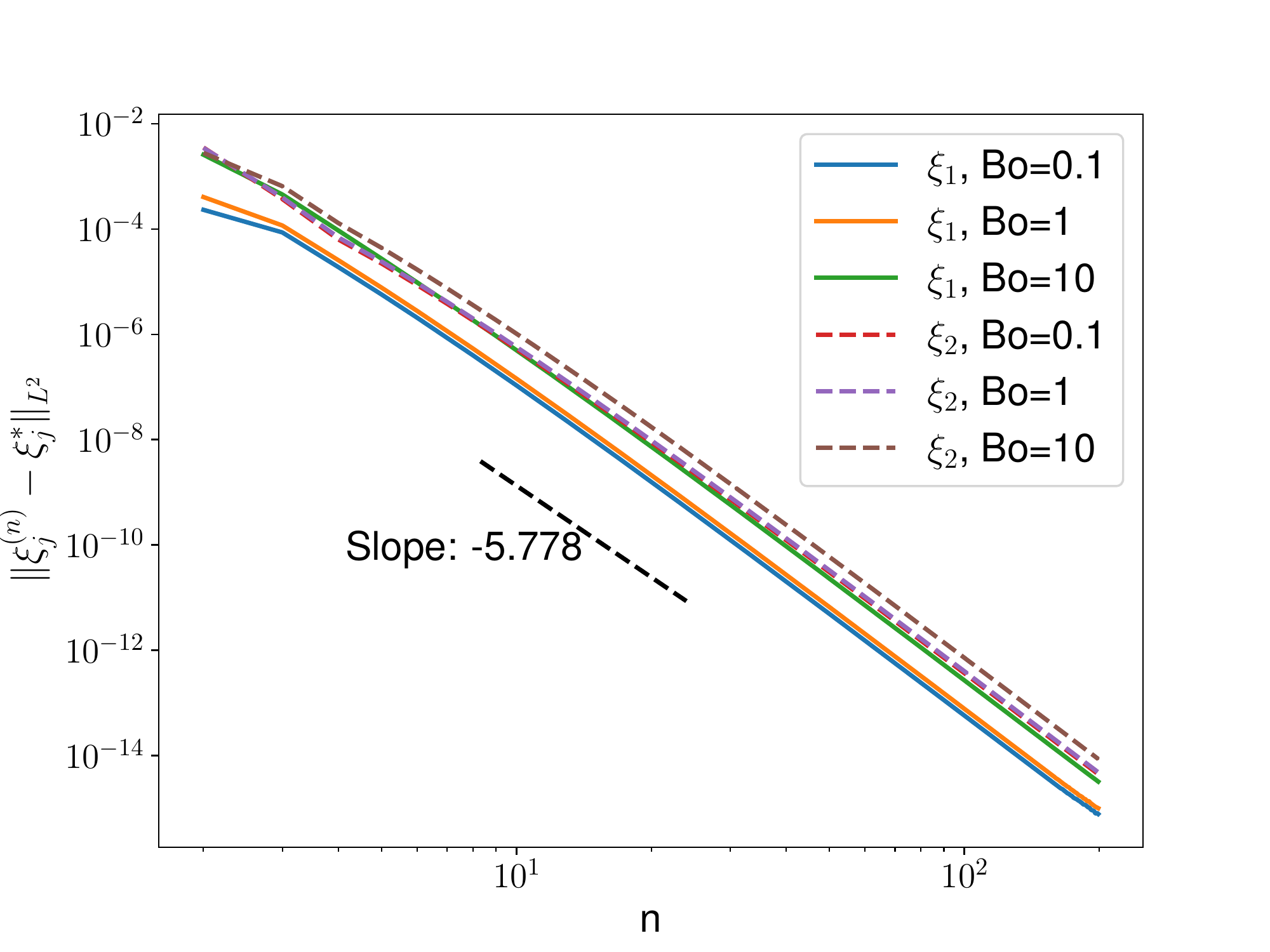}} 
		\caption{Log-log convergence plots of the first (solid) and second (dashed) eigenvalues and their corresponding sloshing profile for $\bond = 0.1, 1, 10$. The true solution is approximated with $n=200$.} 
		\label{fig:IFP_convergence} 
	\end{center}
\end{figure} 

\cref{fig:IFP_convergence_a} illustrates the convergence plots of the first two sloshing frequencies, \ie $j=1,2$, and their corresponding sloshing profiles in \cref{fig:IFP_convergence_b}  for $\bond =\{0.1,1,10\}$. The true solution $(\lambda_j^*,\xi_j^*)$ is the highly resolved ($n=2000$) solution from our method. We observe high rates of convergence for finite $\bond$ for both the sloshing frequencies and the sloshing profiles. We note that the observed rate of convergence for $\xi^{(n)}\to \xi$ is much higher than the a-proiri estimate as discussed above. 

The first three sloshing profiles for $j=1,2,3$ are shown in \cref{fig:IFP_NeumannProfiles} for $\bond=1$ in \cref{fig:IFP_NeumannProfiles_a}, $\bond=1000$ in \cref{fig:IFP_NeumannProfiles_b} and $\bond=\infty$ corresponding to the no surface tension case in \cref{fig:IFP_NeumannProfiles_c}. The sloshing profiles for IFP appear to be unchanged for $\bond <10$. Most interesting is the behavior of the high spot of the fundamental sloshing profile  for moderate $\bond$. This will be further discussed in \cref{sec:HS}, here we simply note that for small $\bond$ the high spot is located on $\del\F$ whereas for large $\bond$ the high spot has moved to interior of $\F$. This phenomenon is further demonstrated for $\bond\in[10^0,10^3]$ in \cref{fig:IFP_HighSpot_b}. 

\begin{figure}[tbhp] 
	\begin{center} 
			\subfloat[$\bond=1$]{\label{fig:IFP_NeumannProfiles_a}\includegraphics[width = 0.325\textwidth,trim = {0.1cm 0.3cm 2cm 1.75cm},clip=true]{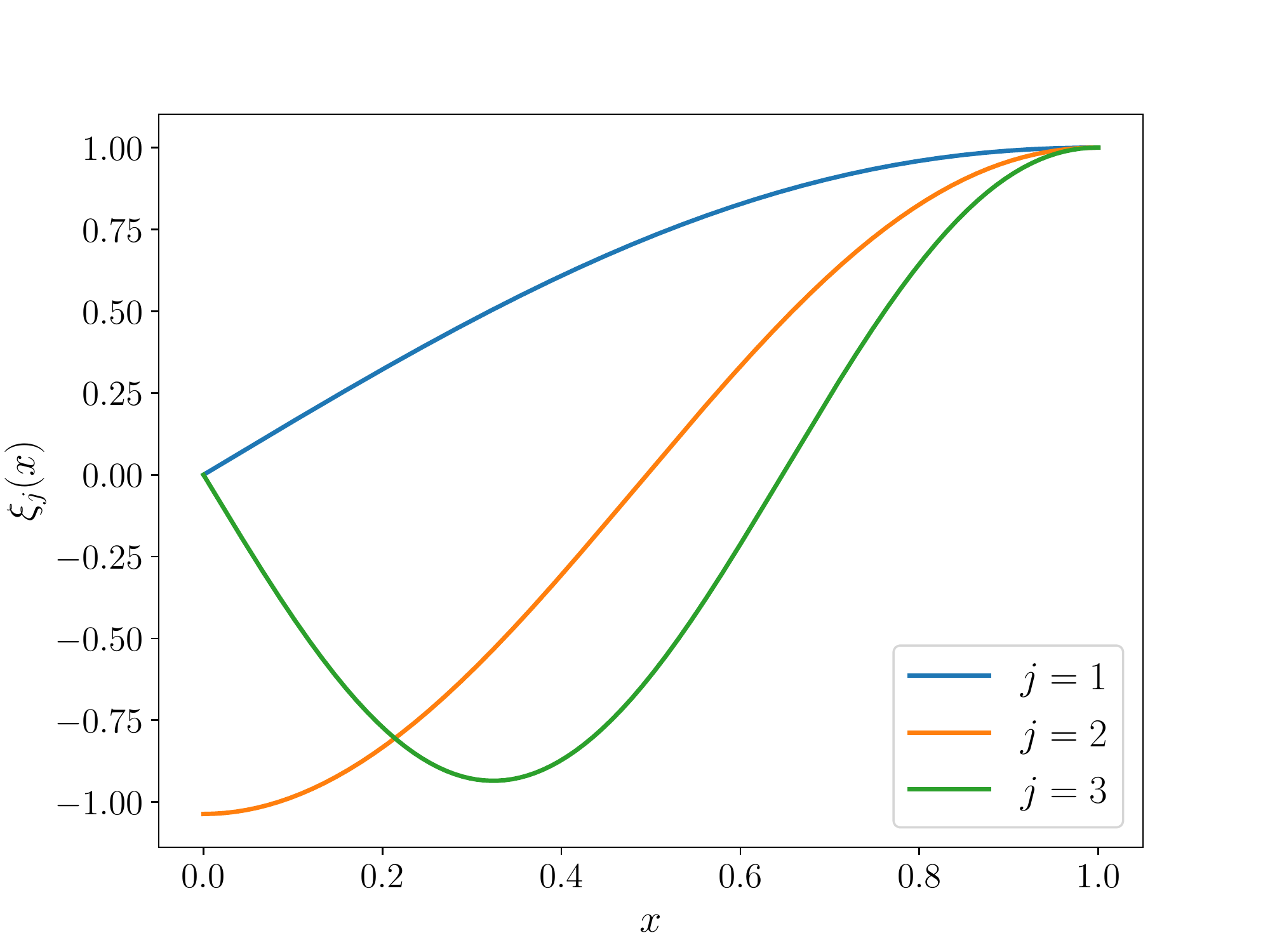}}
			\subfloat[$\bond=1000$]{\label{fig:IFP_NeumannProfiles_b}\includegraphics[width = 0.325\textwidth,trim = {0.3cm 0.3cm 2cm 1.75cm},clip=true]{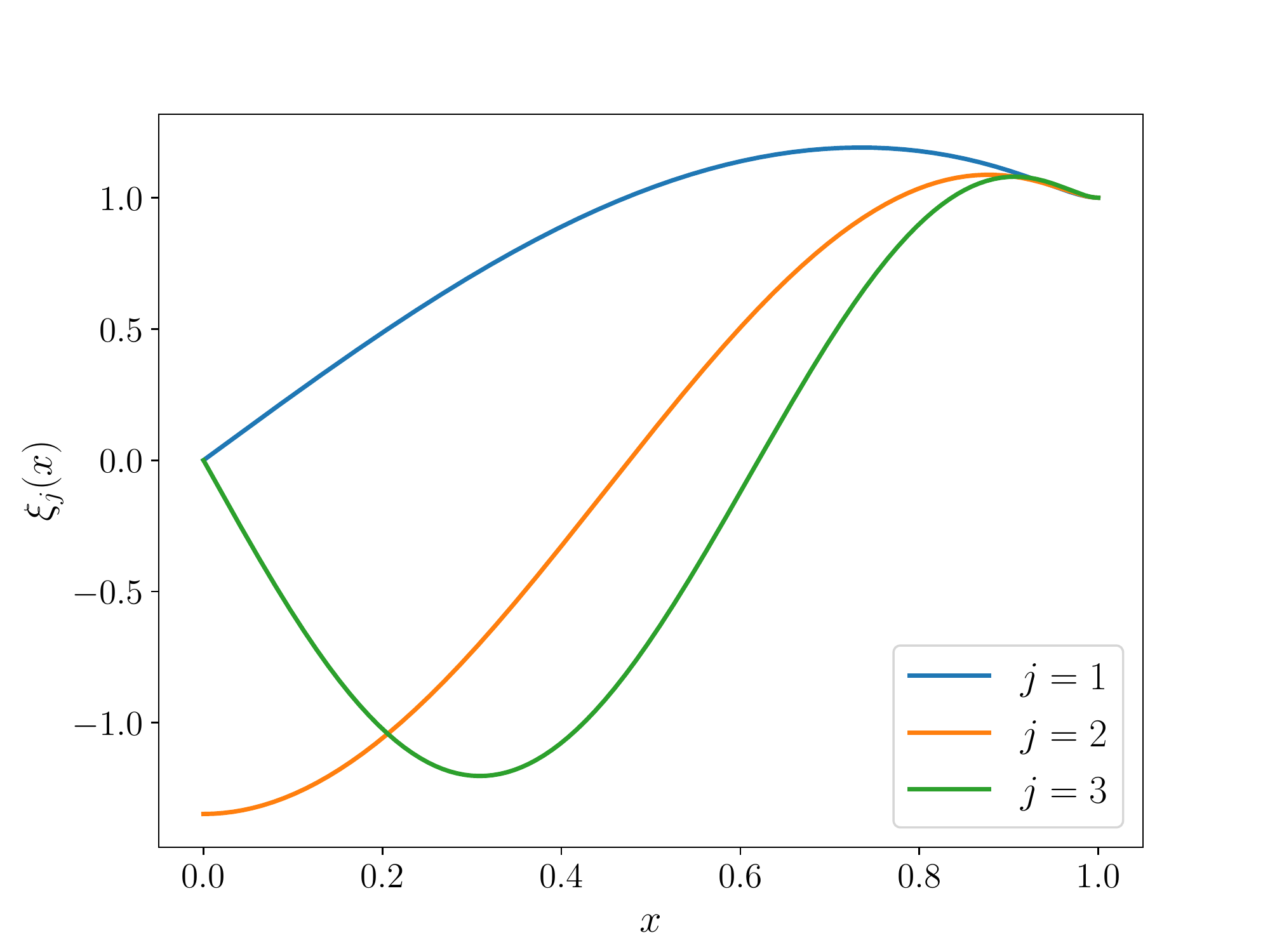}}
			\subfloat[$\bond=\infty$]{\label{fig:IFP_NeumannProfiles_c}\includegraphics[width = 0.325\textwidth,trim = {0.3cm 0.3cm 2cm 1.75cm},clip=true]{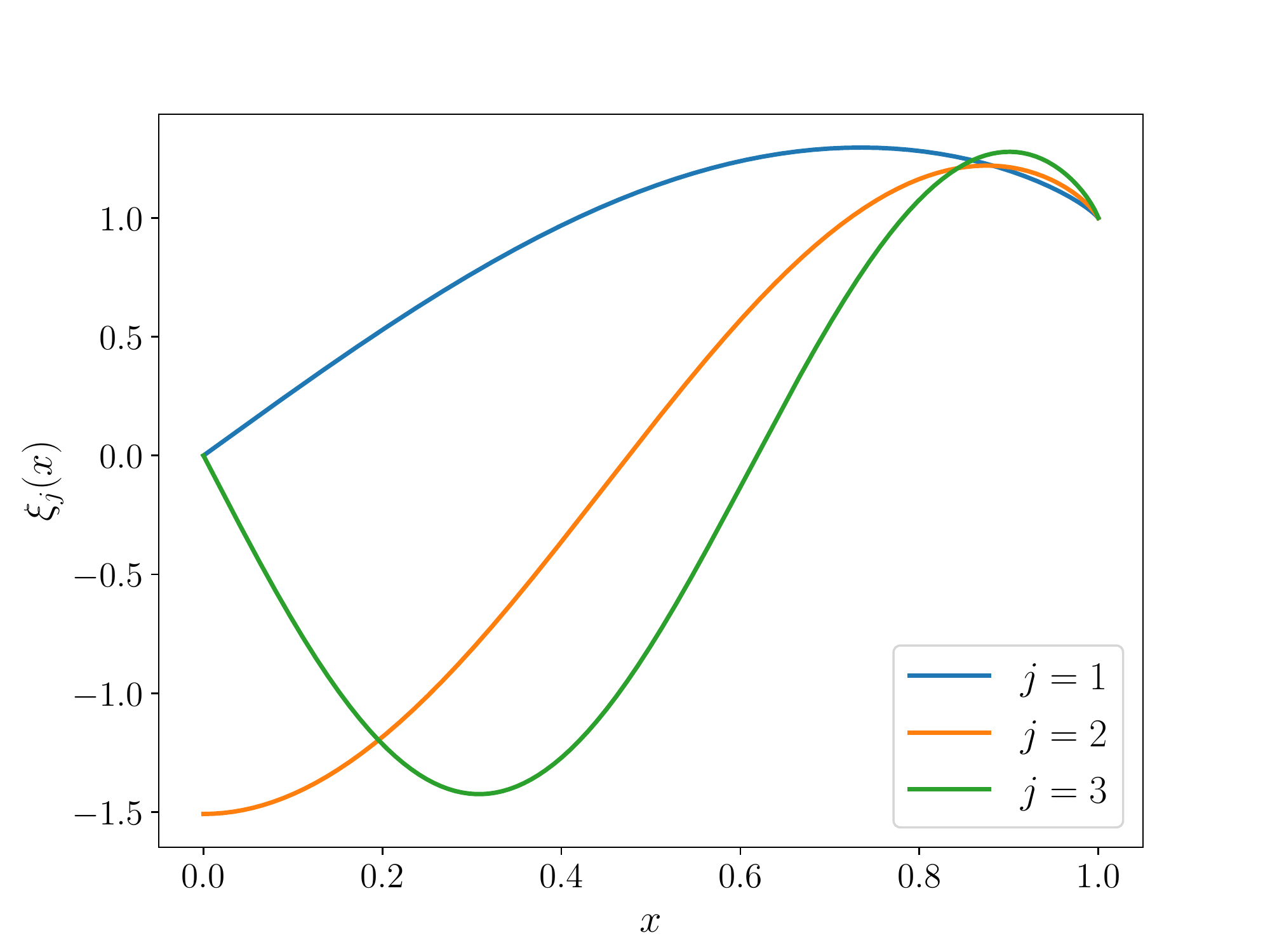}}
		\caption{First three sloshing profiles for the infinite parallel strip for  various $\bond$. Even though $\xi(x)$ is defined on $[-1,1]$ solutions are only plotted on $[0,1]$ since $\xi(x)$ is either even or odd.}
		\label{fig:IFP_NeumannProfiles}
	\end{center} 
\end{figure}  


\section{Circular hole}\label{sec:hole}
For a circular hole, the outward normal to $\F$ in the plane $z=0$ is in the radial direction. In this case, we transform the equations into cylindrical coordinates along the $z-$axis and we look for azimuthal solutions. To construct the integral operator $S$, we consider \cref{eq:Lap} in cylindrical coordinates with $f(r,\theta)=f_m(r)\cos(m\theta)$ and $\B$ the unit disk. Therefore, we make the Ansatz $\Phi(r,\theta,z)=\Phi_m(r,z)\cos(m\theta)$ for $m\geq 0$ with the conditions that $\Phi_0(0,z)$ is bounded and that $\Phi_m(0,z)=0$ for $m\geq 1$. Plugging the Ansatz into \cref{eq:Lap} and canceling $\cos(m\theta)$ yields
\begin{subequations} \label{eq:LapBes} 
	\begin{alignat}{3}  
		\label{eq:LapBes1} \Delta_m\Phi_m +\del_{zz}\Phi_m& = 0 && \ \ \tr{ in } \ \ && \D = \left\{(r, z)\in\R^2\colon r\geq 0\; z < 0\right\}, \\ 
		\label{eq:LapBes2} \del_z\Phi_m & = f_m\chi_{[0,1)} && \ \ \tr{ on }&&\del\D=\{(r,z)\in\R^2\colon r\geq 0\;z=0\}.
	\end{alignat} 
\end{subequations} 
In the above, $\Delta_m=\frac{1}{r}\del_r(r\del_r)-\frac{m^2}{r^2}$ is the Bessel differential operator. We remark that the compatibility condition is automatically satisfied for $m\geq 1$, whereas it becomes $\int_0^1 f_0rdr=0$ for $m=0$. Furthermore, only for $m=0$ can a new solution be obtained by adding a constant to any solution due to the term $\frac{m^2}{r^2}$ in $\Delta_m$. We proceed to construct solutions to \cref{eq:LapBes} using the Hankel transform method. Following \cite{Poularikas}, we define the Hankel transform pair, $F_m(k)=\mathcal{H}\{f\}(k)$, as
\begin{gather*}
    F_m(k)=\int_0^\infty rf(r)J_m(kr)\,dr\quad\text{and}\quad f(r)=\int_0^\infty kF_m(k)J_m(kr)\,dk
\end{gather*}
with $J_m$ being the Bessel functions of the first kind. Using the fact that $\mathcal{H}\{\Delta_m f\}(k)=-k^2\mathcal{H}\{f\}(k)$ \cite{Poularikas}, \cref{eq:LapBes1} becomes an integrable equation for $H_m=\mathcal{H}\{\Phi_m\}$. Looking for bounded solutions as $z\to-\infty$, the general solution has the form $H_m(k,z)=A(k)e^{kz}$. The constant of integration $A(k)$ is obtained using the Neumann boundary condition \cref{eq:LapBes2}. Rearranging reveals that $kA(k)=\mathcal{H}\{f_m\}(k)$. We note that since $f_m(r)$ is compactly supported, the inverse Hankel transform is well defined \cite{Poularikas}. Therefore, we have a solution to the Neumann problem. 
\begin{lemma}\label{lem:IDE_hole} 
	In the case of a radial hole with $f(r,\theta)=f_m(r)\cos(m\theta)$ where $f_m\in L_r^2(0, 1)$, there exists a bounded solution to \cref{eq:Lap} that can be represented as
	\begin{equation*} 
		\Phi(r,\theta,z)=\int_0^\infty\int_0^1 sf_m(s)J_m(ks)e^{kz}J_m(kr)dsdk\cos(m\theta)+c_m=Sf(r,\theta,z)+c_m
	\end{equation*}
	with $c_m=0$ if $m\geq 1$. Furthermore, we have $Sf(r,\theta,z)\to 0$ as $(r,z)\to \infty$, $\theta$ fixed.
\end{lemma}
\begin{proof}
	It is straightforward to show that $\Phi$ is a solution to \cref{eq:LapBes}. 

	Let $G(k) = \int_0^1 \sqrt{s}f_m(s)\sqrt{s}J_m(ks)\, ds$. Using the Cauchy-Schwarz inequality  and \cite[10.22.5]{NIST}, we find 
	\begin{align*} 
		|G(k)|^2  
		\leq \|f_m\|_{L^2_r(0, 1)}^2\cdot\frac{1}{2}\Big[J_m^2(k) - J_{m - 1}(k)J_{m + 1}(k)\Big] . 
	\end{align*} 
	It follows that 
	\begin{align*} 
		\|G(k)\|_{L^4}^4  \le \|f_m\|_{L^2_r(0, 1)}^4\int_0^\infty J_m^4(k) + J_{m-1}^2(k)J_{m+1}^2(k)\, dk < \infty, 
	\end{align*} 
	since $J_m(k)$ is continuous on $[0, \infty)$ and $|J_m(k)|\le k^{-1/3}$ for all $k$ \cite{landau:2000}. Thus $G(k)\in L^4$. Next, using H\"{o}lder's inequality with $p = 4$ and $q = p/(p - 1) = 4/3$, we obtain 
	\begin{align*} 
		|Sf| \le \int_0^\infty |G(k)e^{kz}J_m(kr)|\, dk \le \|G(k)\|_{L^4(0,\infty)}\left\|e^{kz}J_m(kr)\right\|_{L^{4/3}(0,\infty)}. 
	\end{align*} 
	To estimate the second term on the right, we use the same bound as above on $|J_m(k)|$ and a change of variable. We have
	\begin{align*} 
		\left\|e^{kz}J_m(kr)\right\|_{L^{4/3}(0,\infty)}^{4/3} & \le \int_0^\infty \frac{e^{4kz/3}}{(kr)^{4/9}}\, dk  = \left(\frac{3}{4}\right)^{5/9}\frac{1}{(-z)^{5/9}r^{4/9}}\Gamma(5/9), 
	\end{align*} 
	where $\Gamma(u)$ is the Gamma function. Since the last expression on the right goes to zero as $r^2+z^2\to\infty$ with $r>0$ and $z<0$, the claim follows.
\end{proof}

Considering the form of the surface equation \cref{eq:IFP4} for $\xi$ in cylindrical coordinates and looking for azimuthal solutions, we make the same Ansatz $\xi(r,\theta)=\xi_m(r)\cos(m\theta)$. Therefore, for each $m\geq0$, \cref{lem:IDE_hole} gives the form of the integral operator $S$ defined in \cref{def:S} with the restriction to $z=0$ given by
\begin{equation*}
	\widehat{S}\xi(r,\theta)=\int_0^\infty\int_0^1 s\xi_m(s)J_m(ks)J_m(kr)dsdk\cos(m\theta).
\end{equation*}
Therefore, using \cref{thm:Equiv}, plugging into \cref{eq:IDEmain1}, transforming in cylindrical coordinates, using the Ansatz, noting that $M(\xi_m\cos(m\theta))=0$ for $m\geq 1$, and cancelling the cosine, we have the integro-differential equation
\begin{equation} \label{eq:IDE_3D}
	\xi_m - \frac{1}{\bond}\Delta_m\xi_m = \omega^2 \widehat{S}_r \xi_m(r)\quad\text{for}\quad m\geq 1,
\end{equation}
where $\widehat{S}_r\xi_m(r)=\int_0^\infty\int_0^1\xi_m(r)sJ_m(ks)J_m(kr)dsdk$ is the radial part of $\widehat{S}$. In the case $m=0$, we recall that from \cref{eq:IFP6}, we require $ \int_0^1\xi_0 r \, dr = 0.$ We have $\Phi=\Phi_0(r,z)$ and $\xi=\xi_0(r)$ so that application of \cref{thm:Equiv} in cylindrical coordinates yields
\begin{equation} \label{eq:IDE_3D_m0}
	\xi_0 - \frac{1}{\bond}(I-M)\left(\frac{1}{r}(r\xi_0')'\right) = \omega^2 (I-M)\widehat{S}_r \xi_0(r)\quad\text{for}\quad m=0.
\end{equation}
The boundary conditions to \cref{eq:IDE_3D}-\cref{eq:IDE_3D_m0} are $\xi_m'(1)=0$, $\xi_0(0)$ bounded and $\xi_m(0)=0$ for $m\geq 1$.

We now look for weak eigenpairs to \cref{eq:IDE_3D} and \cref{eq:IDE_3D_m0}. Define the Hilbert spaces
\begin{align*}
	V_0 &=\left\{\xi_0\in H^2_r(0,1): \xi_0'(1)=0, \xi_0(0)\text{ is bounded}, \int_0^1r\xi_0\,dr=0\right\}\\
	V_m &=\left\{\xi_m\in H^2_r(0,1): \xi_m'(1)=0, \xi_m(0)=0\right\}, \qquad m\geq 1,
\end{align*}
where the subscript $r$ indicates the weighted Sobolev space with respect to the weight function $w(r)=r$. Consider $m\geq 1$ and let $g_m$ be a test function in $V_m$. Multiplying \cref{eq:IDE_3D} by $g(r) r$, integrating from $0$ to $1$, and using integration by parts along with the boundary conditions, we get the weak formulation
\begin{equation}\label{eq:Weak_3D_m}
	\int_0^1 \left(\xi_m g r + \frac{1}{\bond}\frac{m^2}{r}\xi_m g + \frac{1}{\bond} \xi_m'g' r\right)\, dr = \omega^2 \int_0^1 (\widehat{S}_r\xi_m)(r) g(r) r dr\quad \text{ for } m\geq 1.
\end{equation}
Let $m=0$ and $g_0$ be a test function in $V_0$. Multiplying \cref{eq:IDE_3D_m0} by $g(r) r$, integrating from $0$ to $1$, using integration by parts along with the boundary conditions, and remembering that the $M$ operator produces a constant, we find the weak formulation
\begin{equation}\label{eq:Weak_3D_m0}
	\int_0^1 \left(\xi_0 g r  + \frac{1}{\bond} \xi_m'g' r\right)\, dr = \omega^2 \int_0^1 (\widehat{S}_r\xi_0)(r) g(r) r dr.
\end{equation}
It is obvious that \cref{eq:Weak_3D_m} and \cref{eq:Weak_3D_m0} can be reformulated as a single weak formulation. 
\begin{definition} \label{def:weak3d} 
	We say that $(\omega, \xi_m)\in\R\times V_m$ is a weak sloshing eigenpair of \cref{eq:IDE_3D}, or \cref{eq:IDE_3D_m0} when $m=0$, if the following holds for all $g\in V_m$: 
	\begin{equation}\label{eq:Weak_3D}
		\int_0^1 \left(\xi_m g r + \frac{1}{\bond}\frac{m^2}{r}\xi_m g + \frac{1}{\bond} \xi_m'g' r\right)\, dr = \omega^2 \int_0^1 (\widehat{S}_r\xi_m)(r) g(r) r dr\quad \text{ for }m\geq 0.
	\end{equation}
\end{definition} 


\subsection{Polynomial Approximation} \label{sec:Legendre3d} 
We look for approximate solutions to \cref{eq:Weak_3D} in a finite dimensional space. From the Hankel representation $\int_0^\infty k^2A(k)J_m(kr)dk=f_m(r)$ with $f_m(r)=\omega\xi_m(r)$ along with the fact that $J_m(r) \sim  \frac{ r^m}{2^m \Gamma(m+1)}$ as $r \to 0$  \cite[10.7.3]{NIST} we note that it is therefore necessary that $f_m(r) = O(r^m)$ as $r \to 0$. This fact motivates the use of the radial polynomials $h_j^m(r) = \mu_j^m r^m p_{j-1}^{(0,m)}(2r^2-1)$, with $p_j^{(\alpha,\beta)}(x)$ being the Jacobi polynomials and $\mu_j^m = 2  \sqrt{j + \frac{1}{2}(m-1)}$. We note that these polynomials are orthogonal on $(0,1)$ with respect to the weight function $r$, \ie $\int_0^1 h_i^m(r) h_j^m(r) r dr=\delta_{ij}.$ This fact follows simply from a change of variables, the orthogonality of Jacobi polynomials, and the fact that \cite[18.3.1]{NIST} $\int_{-1}^1 (1+x)^m \left(p_{i-1}^{(0,m)}(x)\right)^2 dx=\frac{2^{m+1}}{2i+m-1}$. Moreover, we note that $\int_0^1 h_j^0 r \, dr = 0$ for $j=2,3,\ldots$.  Now we define $P^{(n)}_m = \mathrm{span}\left\{h_0^m(x), h_1^m(x), \dots, h_n^m(x)\right\}$ same as the infinite parallel strip and analogously define
\begin{equation*}
	W^{(n)}_m = \left\{v\in P^{(n)}_m\colon v'(\pm 1) = 0 \text{ and } \int_{0}^1 vr\, dr = 0 \text{ if }  m=0\right\}.
\end{equation*}

\begin{lemma}
	Suppose $(\omega, \xi_m)$ is a weak sloshing eigenpair of the circular hole IFP, as in \cref{def:weak3d}, and let $\xi_m^{(n)}$ be the orthogonal projection of $\xi_m$ to $W^{(n)}_m$ with respect to the $L^2_r(0,1)$ inner product. When $m\geq 1$, if $\xi_m\in V_m \subset H^2_r(0,1)$ and $\frac{\xi_m}{r} \in L^2_r(0,1)$ then 
	\begin{equation*}
		\left\|\xi_m - \xi_m^{(n)}\right\|_{L^2_r(0,1)} \leq C n^{-1} \left(m\left\| \frac{\xi_m(r)}{r} \right\|_{L^2_r(0,1)} + \left\| \xi_m'(r) \right\|_{L^2_r(0,1)}\right),
	\end{equation*}
	where the constant $C$ is independent of $n$. 
\end{lemma}
\begin{remark} 
	The result holds for $m=0$ simply if $\xi_0 \in V_0$.
\end{remark}
\begin{proof}
	We denote the space of $L^2$ integrable functions on $(-1,1)$ with respect to the Jacobi weight function $w(x)=(1-x)^\alpha(1+x)^\beta$ as $L^2_{\alpha,\beta}$. Now, we note that $\sum_{j=1}^{n-2} \alpha_j h_j^m$ being the orthogonal projection of $\xi(r)$ with respect to the $L_r^2(0,1)$ inner product is equivalent to $\sum_{j=1}^{n-2} \frac{\alpha_j\mu_j^m}{2(2^{m/2})} P_{j-1}^{(0,m)}(x)$ being the orthogonal projection of $\frac{1}{2} \left(\frac{1}{x+1}\right)^{m/2} \xi \left(\sqrt{\frac{x+1}{2}}\right)$ with respect to the $L_{0,m}^2$ inner product. Then, by a change of variables and \cite[Theorem 2.1]{guo2004jacobi} we have that 
	\begin{align*}\small
		\left\|\xi_m - \xi_m^{(n)}\right\|_{L^2_r(0,1)} &= \left\|\frac{1}{2} \left(\frac{1}{x+1}\right)^{m/2} \xi_m \left(\sqrt{\frac{x+1}{2}}\right)
		-\sum_{j=1}^{n-2} \frac{\alpha_j\mu_j^m}{2(2^{m/2})} p_{j-1}^{(0,m)}(x)\right\|_{L^2_{0,m}}, \\
		& \leq Cn^{-k} \left \| \frac{d^k}{dx^k} \left( \frac{1}{2} \left(\frac{1}{x+1}\right)^{m/2} \xi_m \left(\sqrt{\frac{x+1}{2}}\right) \right) \right\|_{L^2_{1,m+1}}.
\end{align*}
To relate back to the $L^2_r(0,1)$ inner product we consider the specific case with $k=1$, change variables, and use the triangle inequality to determine 
\begin{equation*} \small
	\left \| \frac{d}{dx} \left( \frac{1}{2} \left(\frac{1}{x+1}\right)^{m/2} \xi \left(\sqrt{\frac{x+1}{2}}\right) \right) \right\|_{L^2_{1,m+1}} \leq m \left\| \frac{\xi_m(r)}{r} \right\|_{L^2_r(0,1)} + \left\| \xi_m'(r) \right\|_{L^2_r(0,1)},
\end{equation*}
and the result follows. 
\end{proof}

To strongly enforce the boundary conditions we define  $q_j^m = h_j^m + \beta_j h_{j+1}^m$, where the constant $\beta_j$ is such that $(q_j^m)'(1) = 0$. We know \cite[18.6.1]{NIST} that $P_{n}^{(0,m)} = \frac{(1)_n}{n!} = 1$ and that $\frac{d}{dr} P_{j-1}^{(0,m)}(1) = \frac{(j-1)(j+m)}{2}$, which comes from evaluating the Jacobi polynomial differential equation at 1, yielding $\left(h_j^m\right)'(1) = \mu_j^m m + 2 \mu_j^m (j-1)(j+m)$. Therefore, we have that $\beta_j = -\frac{\mu_j^m \left(m + 2(j-1)(j+m) \right)}{\mu_{j+1}^m \left(m + 2j (j+m+1)\right)}$. Note that if Dirichlet boundary conditions were prescribed, one would simply define $\beta_j^m = -\frac{\mu_j^m}{\mu_{j+1}^m}$.

Therefore, the discrete weak formulation of \cref{eq:Weak_3D} is to find $(\omega, \xi_m^{(n)}) \in \R \times W^{(n)}$ such that for all $g\in W_m^{(n)}$: 
\begin{equation} \label{eq:IFPhole_discrete} \small
	\begin{split}
		\int_0^1 &\left(\xi_m^{(n)}(r) g(r) + \frac{1}{\bond} \frac{m^2}{r}\xi_m^{(n)}(r)g(r)+\frac{1}{\bond}\left(\xi_m^{(n)}(r)\right)'g'(r)r\right)\, dr\\
		&= \omega^2\int_0^1\int_0^\infty\int_0^1\xi_m^{(n)}(r)sJ_m(ks)J_m(kr) g(r) r ds dk dr .
	\end{split}
\end{equation} 
\begin{theorem}\label{thm:discrete_hole}
	If $(\omega,\xi_m^{(n)})\in\R\times W^{(n)}$ solves the radial discrete IFP \cref{eq:IFPhole_discrete}, then $(\omega,\xi^{(n)})$ solves the generalized eigenvalue problem for $i=1,2,\ldots,n$
	\begin{equation} \label{eq:IDEhole_numerics} \small
		\sum_{j=1}^{n} a_j^m \left(M_{ij}^m + \frac{1}{\bond}K_{ij}^m\right) = \omega^2 \sum_{j=1}^{n} a_j^m L_{ij}^m ,
	\end{equation}
	where $M_{ij}^m =  \int_0^1 q_j^m q_i^mr\, dr$ is the mass matrix,  $K_{ij}^m=\int_0^1 \left( \frac{m^2}{r}q_j^mq_i^m + (q_j^m)'(q_i^m)' r\right)\, dr $ is the stiffness matrix, and $L_{ij}^m=\int_0^1 \int_0^\infty \int_0^1 q_j^m(s) J_m(ks) s\, ds J_m(kr)\, dk q_i^m(r) r dr.$ For the $m=0$ case this simply needs to be edited to $\ds \xi_m = \sum_{j=2}^{n} a_j^m q_j^m$ and it only holds for $i=2,\ldots,n$. Moreover, 
	\begin{align*} \small
		M_{ij}^m &= \left(1 + \beta_i^2\right)\delta_{ij} + \beta_i \delta_{i+1,j} + \beta_{i-1}\delta_{i-1,j}, \\
		K_{ij}^m &= -\mu_j^m \mu_{j+1}^m \left(\mu_j^{m+1}\right)^2 \beta_j^m \delta_{ij}, \\
		L_{ij}^m &= \tilde{L}_{ij}^m + \beta_i \tilde{L}_{i+1,j}^m + \beta_j \tilde{L}_{i,j+1}^m + \beta_i\beta_j\tilde{L}_{i+1,j+1}^m,\\
		\tilde{L}_{ij}^m &= \frac{\mu_j^m \mu_i^m}{4 \pi \left((j+i+m-1)^2-\frac{1}{4}\right)\left(\frac{1}{4} - (j-1)^2\right)}.
	\end{align*}
\end{theorem}
\begin{proof}
Again, we write $ \xi_m^{(n)} = \sum_{j=1}^{n} a_j^m q_j^m$ for $i=1,2,\ldots,n$ and we choose $g=q_i^m$ in \cref{eq:Weak_3D} to obtain the integral form of the matrices. The lengthy calculations to obtain the coefficients are given in \cref{sec:App3dMatrices}.
\end{proof}


\subsection{Numerical Results}\label{sec:numerics_hole}
Similar to the infinite parallel strip we are left to solve a generalized eigenvalue problem of the form \cref{eq:IFP_GEP}. In the circular hole problem, the mass matrix $\mathbf{M}$ is tridiagonal, the stiffness matrix $\mathbf{K}$ is diagonal, and the matrix $\mathbf{L}$ is dense. Again, it follows from \cref{thm:Equiv} that the eigenvalues $\lambda_j$ of \cref{eq:IFP_GEP} approximate the first $n$ of the natural sloshing frequencies of IFP squared. 

The first three eigenvalues, for $m=0,1,10$, of \cref{eq:IFP_GEP} for the circular hole for $\bond = 1, 10, 50, \infty$ with $n = 200$ are shown in \cref{tab:Neu3DEigvals}. To validate our solution to the circular hole IFP for $\bond = \infty$ we compared the eigenvalues to those found in Miles \cite{miles:1972} for $m=0,1,2,3$ and note that for all eigenvalues reported our results agree to the accuracy given by Miles (6 digits). Miles \cite{miles:1972} compares their results to that of Henrici et al. \cite{henrici:1970} and finds their results to be as accurate or more accurate. Miles \cite{miles:1972} reports that results are given for $n=16$. The current work presents results with $n=64$ to recover the same, or better, accuracy. This is necessary due to the enforcement of boundary conditions in the current work, which is not present in the $\bond = \infty$ problem and for this reason not considered in \cite{miles:1972}. In the circular hole problem, with $m=1$ we have that $\ds \frac{\lambda_1(\bond = 1)}{\lambda_1(\bond = \infty)} = 4.510$. That is, when the force due to surface tension is comparable to the gravitational force the fundamental sloshing frequency is increased to more than 400\% of the value when surface tension is negligible. 

\begin{table}[tbhp]
	\centering
    	\begin{tabular}{|c|c|c|c|c||c|} \hline
                        				    &           & $\bond=1$ & $\bond=10$ & $\bond=50$ & $\bond=\infty$ \\ \hline
		\multirow{3}{*}{$m=0$}  & $j=1$ &   64.9935 &  10.2253   &   5.3528   & 4.1213         \\ \cline{2-6} 
                        				    & $j=2$ &  369.5505 &  43.5842   &  14.6085   & 7.3421         \\ \cline{2-6} 
                        				     & $j=3$ & 1100.4019 & 119.5281   &  32.3396   &  10.5171       \\ \hline
		\multirow{3}{*}{$m=1$}  & $j=1$ &   12.4245 &   3.7758   &   2.9854   & 2.7548         \\ \cline{2-6} 
                        				     & $j=2$ &  172.4077 &  22.5719   &   9.2589   & 5.8921         \\ \cline{2-6}
                        				     & $j=3$ &  665.6328 &  74.7182   &  22.1968   & 9.0329         \\ \hline
		\multirow{3}{*}{$m=10$} & $j=1$ & 1971.9593 & 209.8054   &  53.1458   & 13.5734        \\ \cline{2-6} 
                        				      & $j=2$ & 4740.2111 & 489.8354   & 112.0301   & 17.4838        \\ \cline{2-6}
                        				      & $j=3$ & 8611.8754 & 880.1861   & 192.9290   & 21.0661        \\ \hline
    	\end{tabular}
	\caption{Eigenvalues $j=1,2,3$ for the azimuthal modes $m=1,2,3$ and for $\bond=\{1,5,10,20,50,\infty\}.$}
	\label{tab:Neu3DEigvals}
\end{table}

\cref{fig:IFP3Dm0_convergence_a} illustrates the convergence plots of the sloshing frequencies $j=1,2$ with $m=1$ and their corresponding sloshing profile in \cref{fig:IFP3Dm0_convergence_b} for $\bond =\{0.1,1,10\}$. The true solution $(\lambda_j^*,\xi_j^*)$ is the highly resolved ($n=2000$) solution from our method. Convergence for the sloshing profiles is measured in the $L^2$-norm. In both the sloshing frequencies and sloshing profiles we observe a high rate of convergence. This rate of convergence is unaffected by the choice of $m$ so only the results for $m=1$ are shown.

\begin{figure}[tbhp]  
	\begin{center}
			\subfloat[$\lambda_1,\lambda_2$ convergence]{\label{fig:IFP3Dm0_convergence_a}\includegraphics[width = 0.49\textwidth]{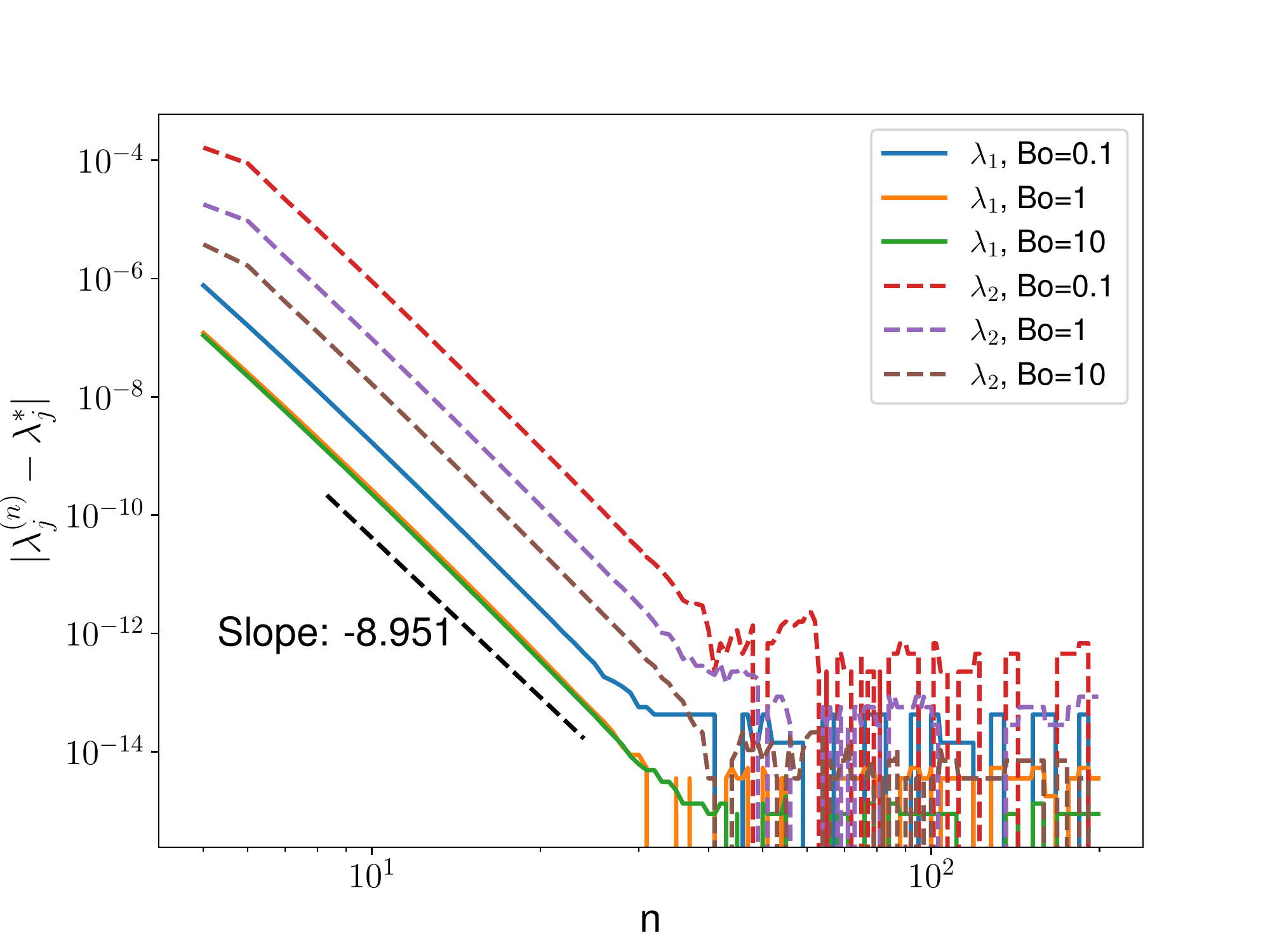}}
			\subfloat[$\xi_1,\xi_2$ convergence]{\label{fig:IFP3Dm0_convergence_b}\includegraphics[width = 0.49\textwidth]{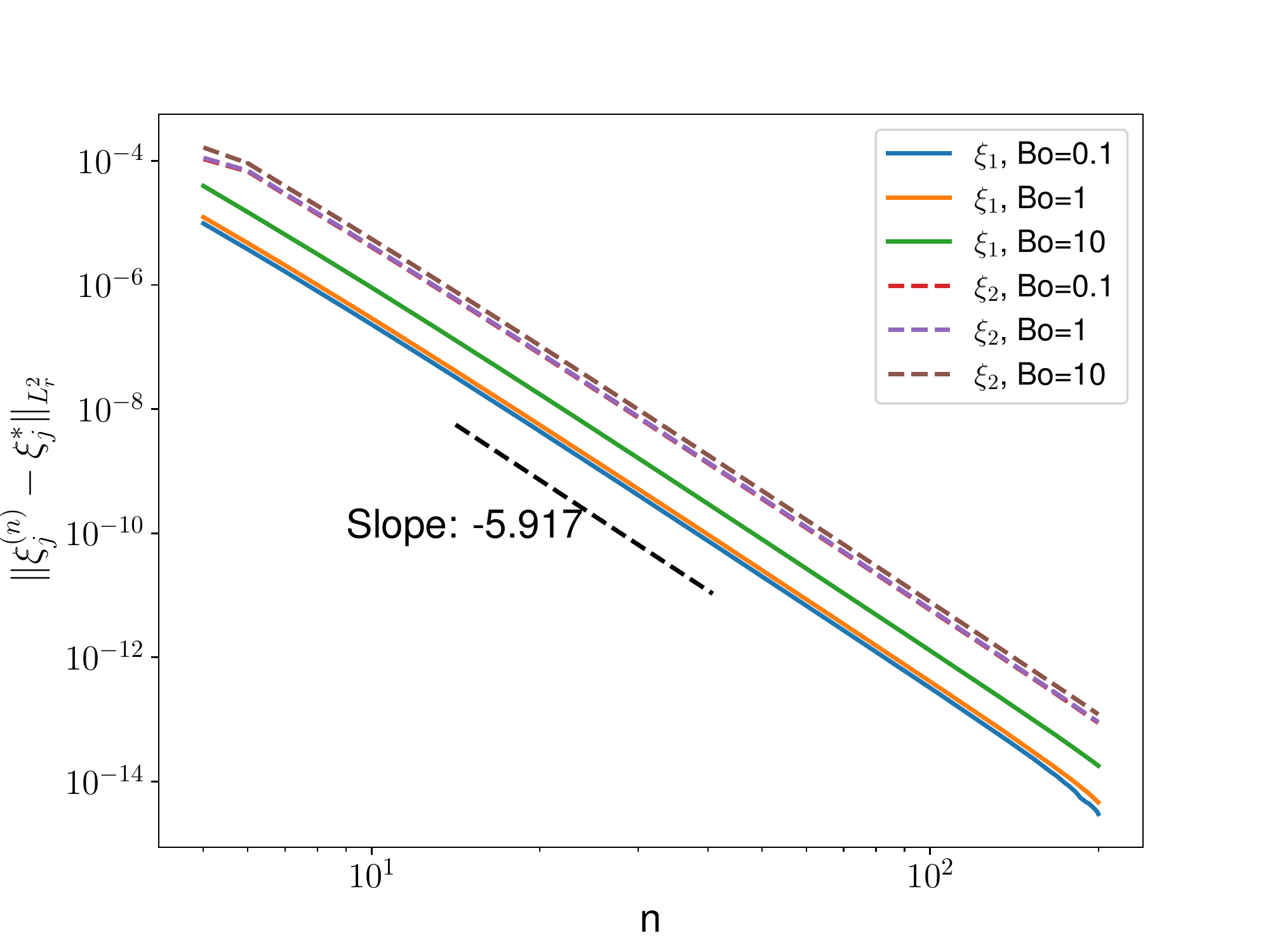}}
		\caption{Log-log convergence plots of the first (solid) and second (dashed) eigenvalues and their sloshing profile for $\bond = 0.1, 1, 10$ with $m=1$. The true solution is obtained with $n=200$.} 
		\label{fig:IFP3Dm0_convergence} 
	\end{center}
\end{figure} 

The first three sloshing profiles, for $m=0,1,2$, for $\bond=\{1,1000,\infty\}$ are shown in \cref{fig:IFP3d_NeumannProfiles}. Again, the sloshing profiles for the circular hole appear to be unchanged for $\bond <10$. We observe the same phenomenon as observed in the infinite parallel strip case that for small $\bond$ the high spot is located on $\del\F$ whereas for large $\bond$ the high spot has moved to interior of $\F$. This can also be seen for $\bond\in[10^0,10^3]$ in \cref{fig:IFP_HighSpot_c} and it will be discussed further in \cref{sec:HS}.

\begin{figure}[tbhp] 
	\begin{center}
			\subfloat[$\bond=1$]{\label{fig:IFP3d_NeumannProfiles_a}\includegraphics[width = 0.32\textwidth,trim = {0.7cm 0.3cm 2cm 1.75cm},clip=true]{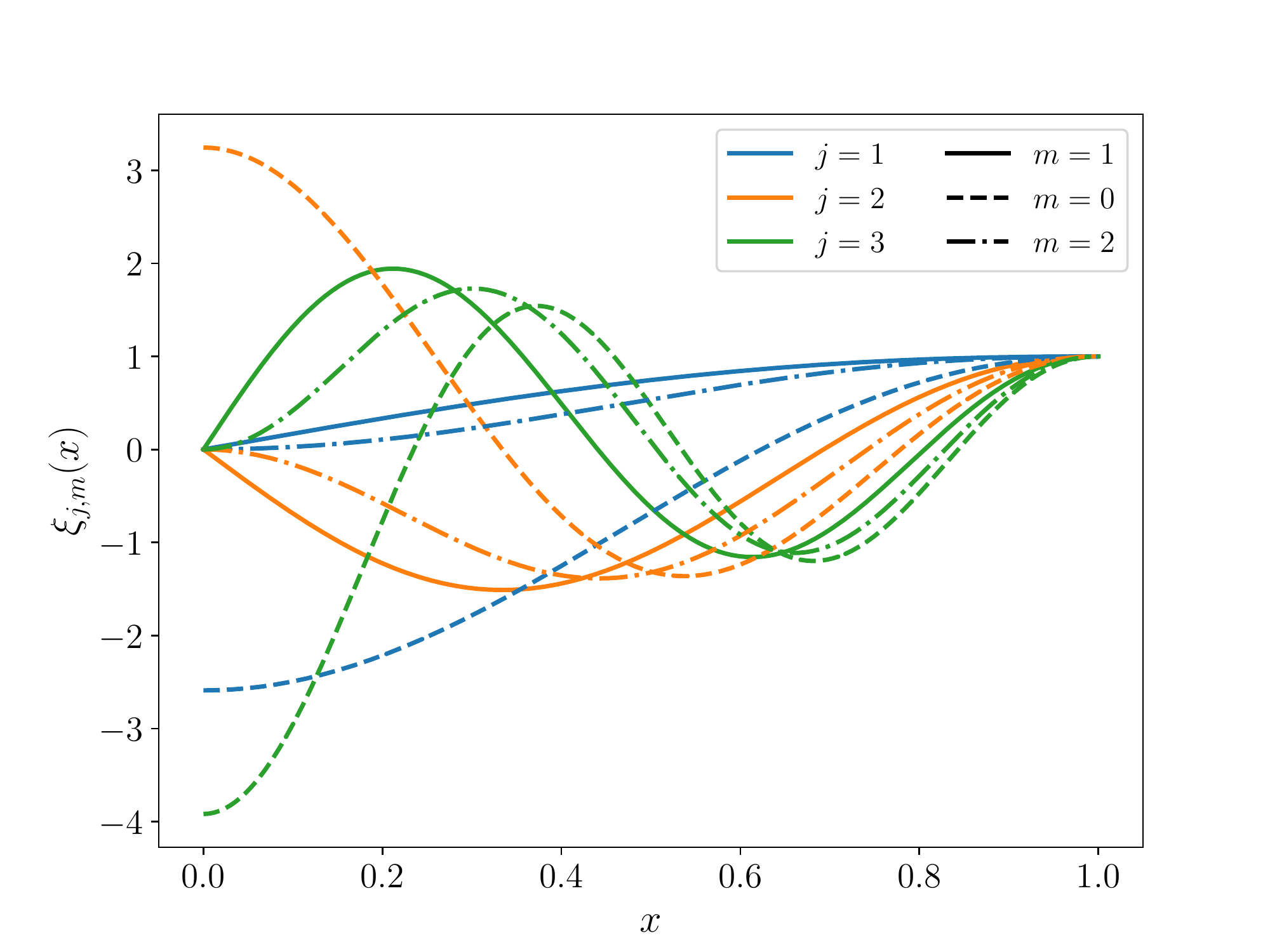}}
			\subfloat[$\bond=1000$]{\label{fig:IFP3d_NeumannProfiles_b}\includegraphics[width = 0.32\textwidth,trim = {0.7cm 0.3cm 2cm 1.75cm},clip=true]{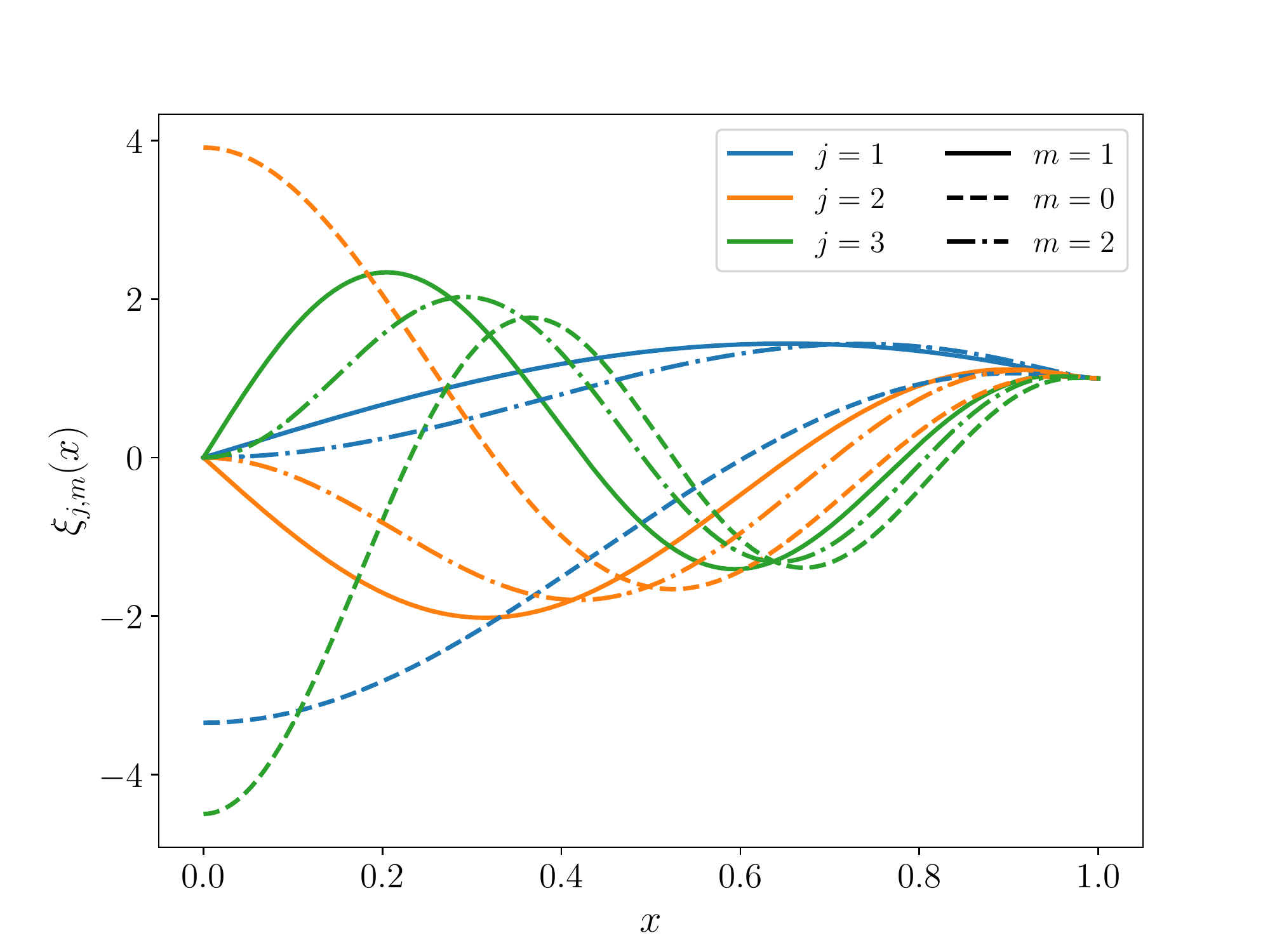}}
			\subfloat[$\bond=\infty$]{\label{fig:IFP3d_NeumannProfiles_c}\includegraphics[width = 0.32\textwidth,trim = {0.7cm 0.3cm 2cm 1.75cm},clip=true]{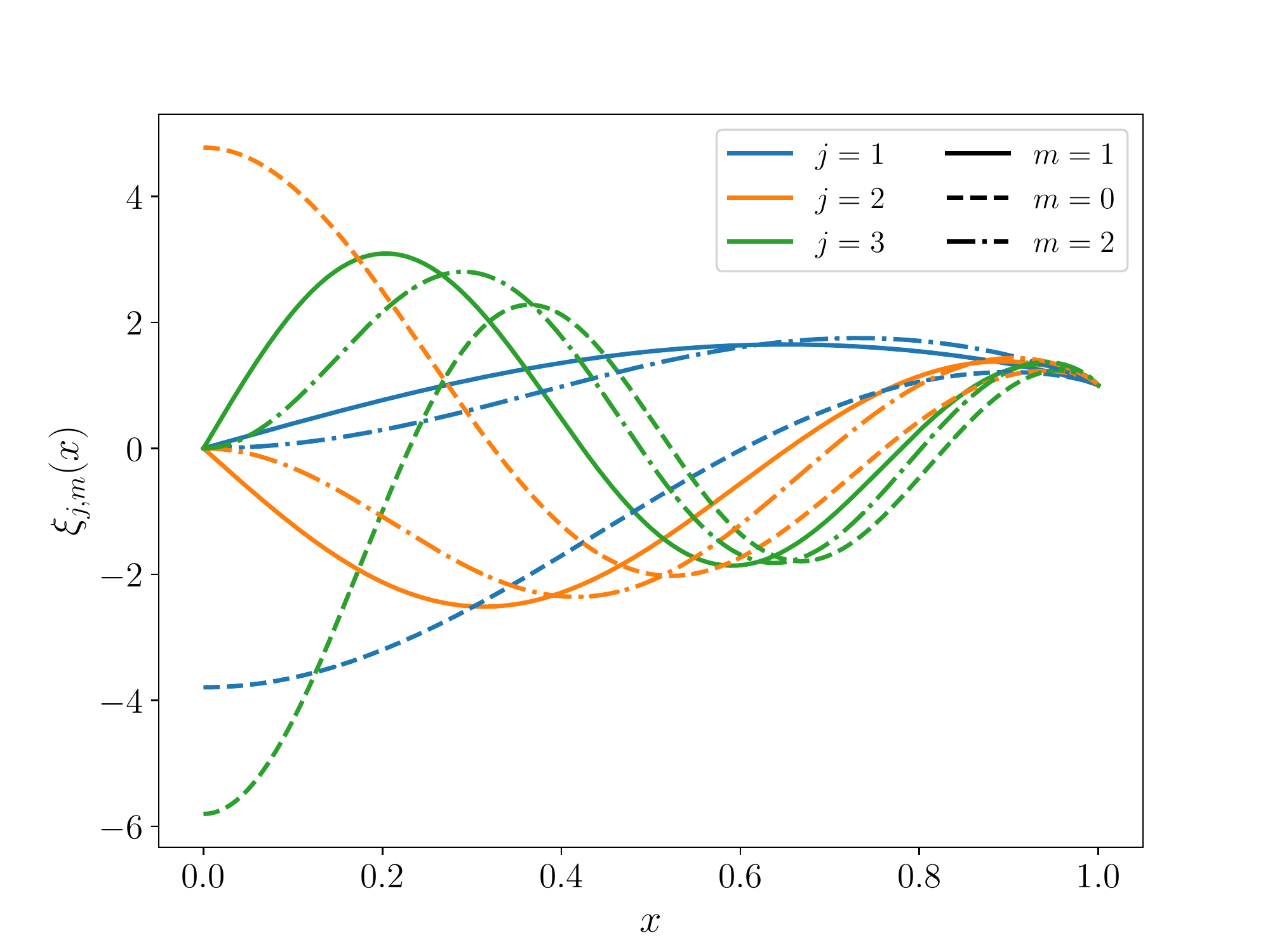}}
		\caption{The first three sloshing profiles $j=\{1,2,3\}$ (blue, orange, green) for the IFP with a radial hole for $m=0$ dashed, $m=1$ solid and $m=3$ dot dashed and for varying $\bond$.} 
		\label{fig:IFP3d_NeumannProfiles}
	\end{center} 
\end{figure}  


\section{High spot justification for a radial hole} \label{sec:HS}
In this section, we offer a justification for the main result that with sufficiently strong surface tension the high spot of the sloshing profile can be moved to the boundary of the free surface. We recall that previous results in the absence of surface tension, \ie  $\bond\to\infty$, show that the high spot for a radial aperture is always inside the domain \cite{kulczycki:2009}. While we focus on the fundamental sloshing height, which corresponds to $m=1$, $j=1$, we demonstrate that the method described below to find the critical $\bond$ such that the high spot is on the boundary applies for any $m\geq 1$ when such a $\bond$ exists. \cref{fig:IFP_HighSpot_c} plots $\bond$ in terms of the location of the high spot for a radial hole, $m=1$, and $j=1$. It is obtained using Newton's method for $\xi_1'(r)$ and it shows that for $\bond>\bond^*=4.63462$, the high spot is located in the interior and asymptotes to the known location $x=0.650312$ for $\bond\to\infty$ and that for $\bond<\bond^*$ the high spot is on the boundary. 

To justify the numerical observation that for sufficiently small $\bond$ the high spot is on the boundary, we remark that $r=1$ is always a critical point of $\xi_m(r)$ in $[0,1]$ as a consequence of the Neumann boundary condition. Furthermore, for $m\neq 0$, if $r=1$ is a local minimum, then we must have by continuity of $\xi_m(r)$ that there exists at least one local maximum in $(0,1)$ since $\xi_m(0)=0$ and $\xi_m$ is scaled such that $\xi_m(1)=1$. Thus we start by finding the concavity of $\xi_m$ at $r=1$ and by determining $\bond=\bond^*$ for which $\xi_m''(1)=0$, \ie where $\xi''_m(1)$ changes sign. To do so, we evaluate \cref{eq:IDE_3D} for $m\geq 1$ at $r=1$ to get $\xi_m''(1) = m^2 + \bond\left(1 - \omega^2\widehat{S}_r \xi_m(1)\right)$. Here, we have used the fact that $\xi_m'(1) = 0$ and that $\xi_m(1)=1$. Plugging in $\xi_m''(1) = 0$ and noting that $(\omega,\xi_m)$ depends on $\bond$ we define 
\begin{equation*} \small 
	T(\bond) = \frac{m^2}{\omega^2\widehat{S}_r \xi_m(1)-1},
\end{equation*}
and seek fixed points $\bond^* = T(\bond^*)$. The map $T(\bond)$ with $m=1$ is plotted in \cref{fig:T_Bo_a} in blue and the fixed points are the intersection with the black dashed line $y=\bond$. Intuitively, to find the fixed points of $T(\bond)$ one would consider the iterative scheme $\bond_{n+1} = T(\bond_n)$. However, it is obvious from the slope at the fixed point in \cref{fig:T_Bo_a} that these iterations will diverge. Additionally, we observe that $T(\bond)$ has a discontinuity where the denominator is zero. It follows naturally to instead consider the reciprocal map $1/T(1/x)$, where $x=1/\bond$, shown in \cref{fig:T_Bo_b} for $m=1$. This eliminates the observed discontuinity, but $x_{n+1} = 1/T(1/x_n)$ would still diverge away from the fixed point. 
\begin{figure}[tbhp]
	\begin{center}
			\subfloat[$T(\bond)$]{\label{fig:T_Bo_a}\includegraphics[width = 0.32\textwidth,trim={1cm 0.2cm 2cm 1.6cm},clip]{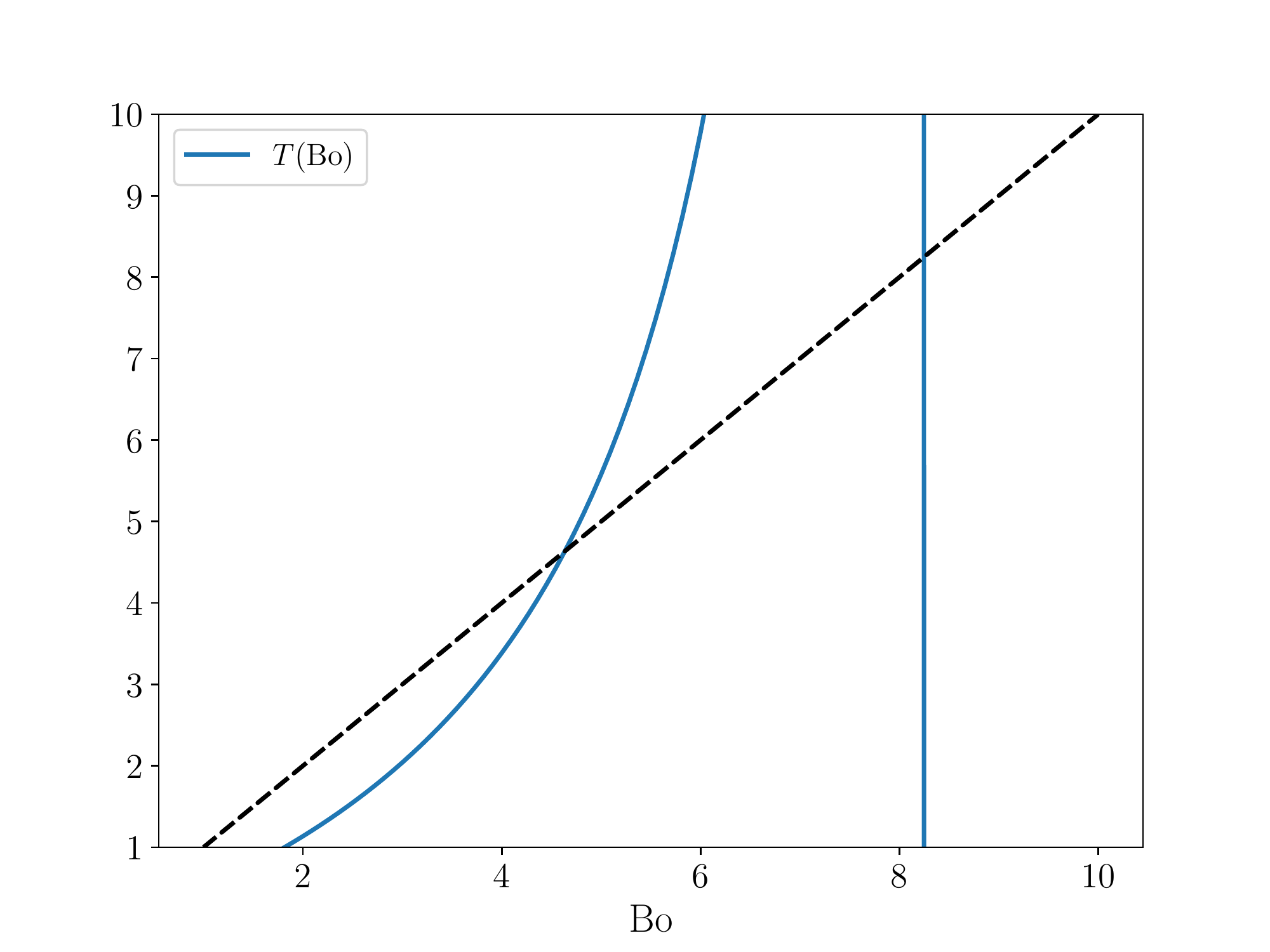}}
			\subfloat[$\frac{1}{T\left(\frac{1}{x}\right)}$, $x=\frac{1}{\bond}$]{\label{fig:T_Bo_b}\includegraphics[width = 0.32\textwidth,trim={1cm 0.2cm 2cm 1.6cm},clip]{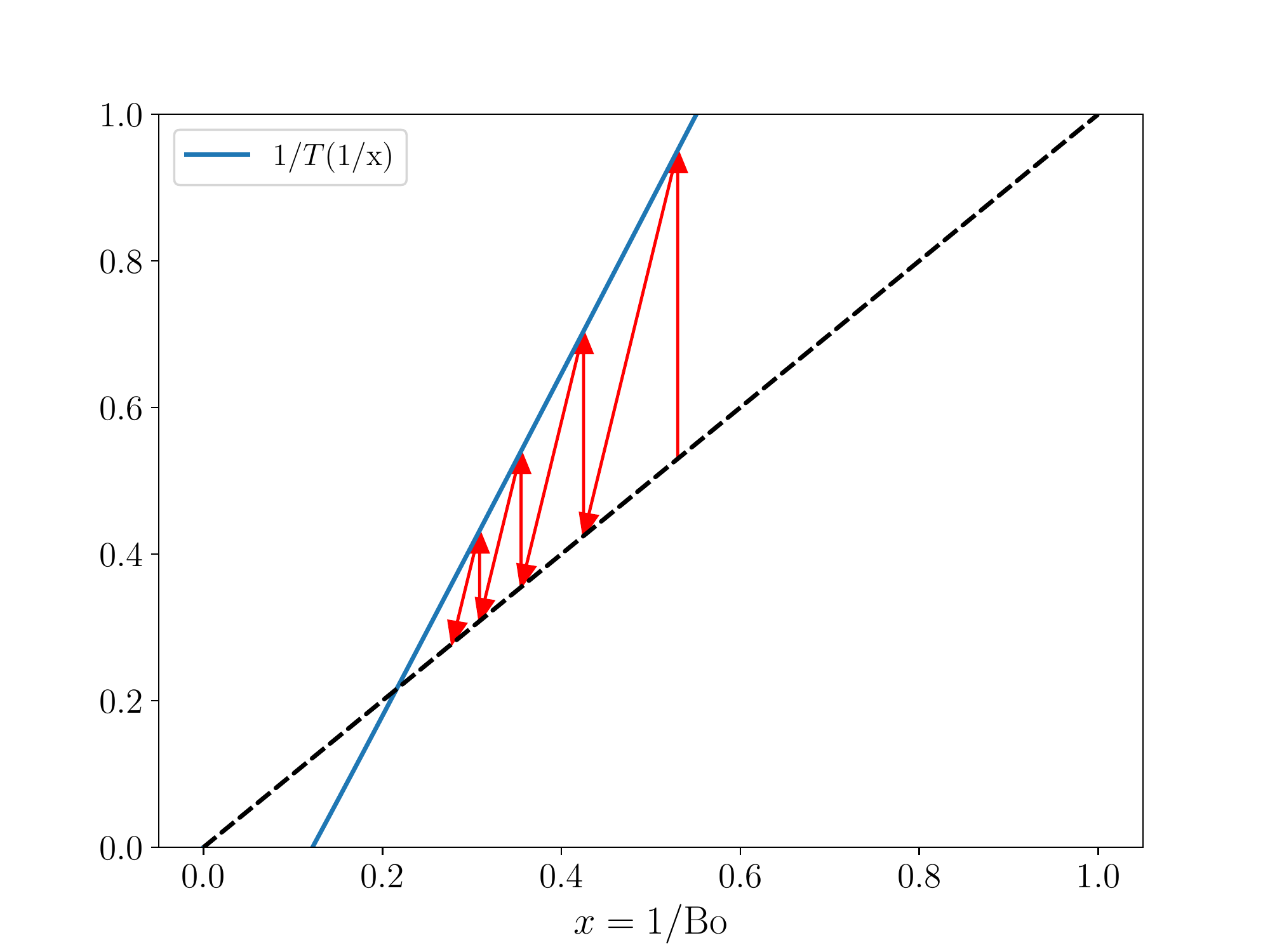}}
			\subfloat[$\widetilde{T}(x)$, $x=\frac{1}{\bond}$]{\label{fig:T_Bo_c}\includegraphics[width = 0.32\textwidth,trim={1cm 0.2cm 2cm 1.6cm},clip]{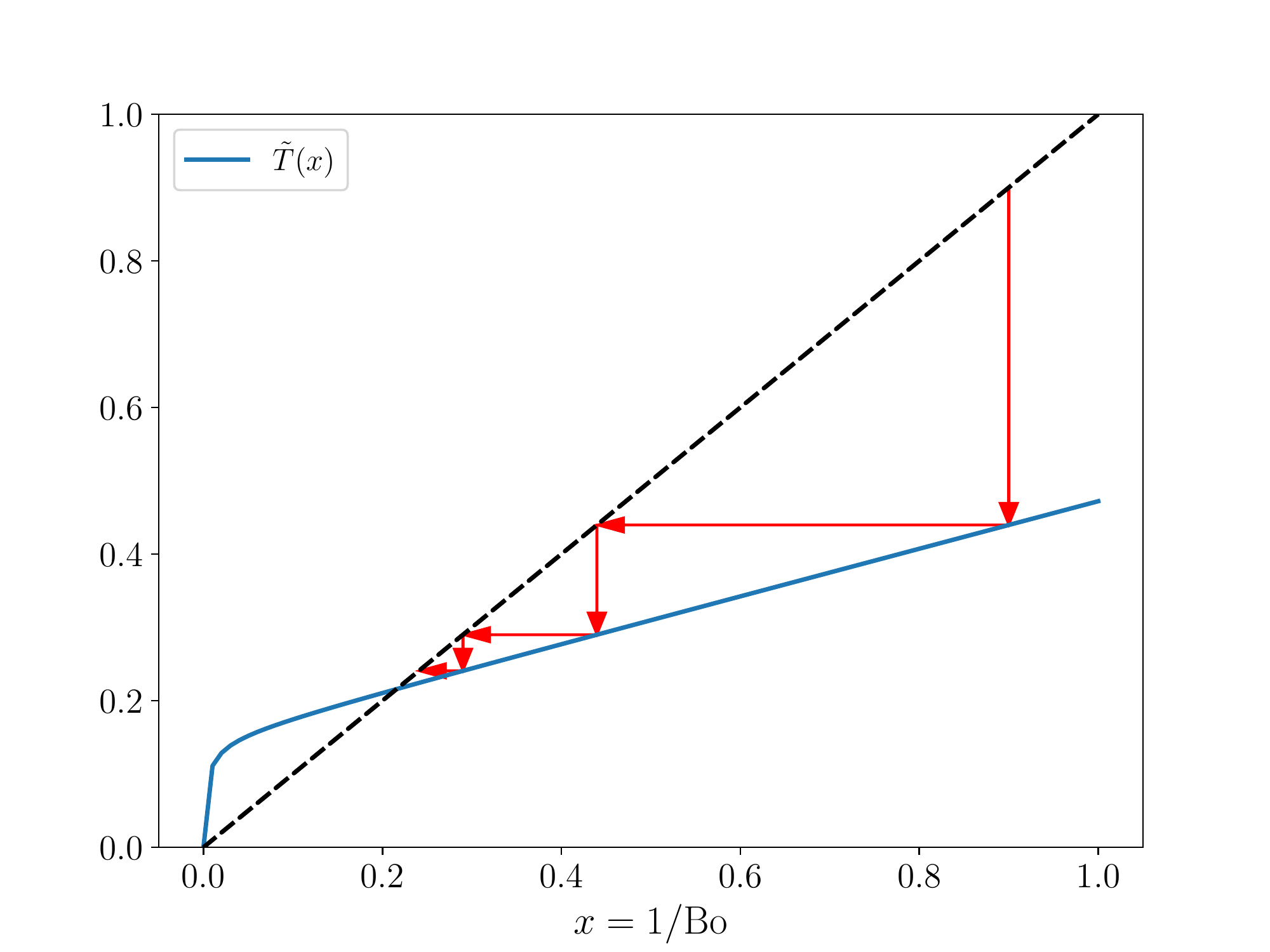}}
		\caption{(a) $T(\bond)$ plotted for $\bond \in [1,10]$ and $m=1$ where fixed points are at the intersection with the black dashed $y=\bond$. We note the discontinuity of $T(\bond)$. (b) Reciprocal map $1/T(x)$ with $x=1/\bond$ together with the modified fixed point iteration with slope $\alpha=3$ (c) New map $\widetilde{T}(x)$ with $x=1/\bond$ (blue) plotted for $x\in (0,1]$, or $\bond \in [1,\infty)$, $m=1$, and $\alpha=3$ where we note a fixed point for finite $\bond$ at the intersection with the dashed line $y=x$ together with the fixed point iteration.}
		\label{fig:T_Bo}
		\end{center}
\end{figure}
Therefore, instead of iterating along a horizontal line from the map $1/T(1/x_n)$ to the line $x_{n+1} = x_n$ we iterate along a line passing through $(x_n,1/T(1/x_n))$ with slope $\alpha$ to the line $x_{n+1}=x_n$, as illustrated by the red arrows in \cref{fig:T_Bo_b}. We note that $\alpha = 0$ simply corresponds to a horizontal line which returns the classical cobweb scheme. Additionally, we see that if $\alpha = 1$ then the line we are iterating along is parallel to the line $x_{n+1}=x_n$ and the scheme is not sensible. With this idea, we define a new map 
\begin{equation} \label{eqn:Ttilde} \small
	\widetilde{T}(x) = \frac{1}{1 - \alpha}\left(\frac{1}{T\left(\frac{1}{x}\right)} - \alpha x\right), 
\end{equation}
shown in \cref{fig:T_Bo_c} with $m=1$, $\alpha=3$, and now consider the iterative scheme $x_{n+1}=\widetilde{T}(x_n)$, as illustrated by the red arrows in \cref{fig:T_Bo_c}. The following lemma follows by simple computations. 

\begin{lemma}\label{lem:FixedPoints}
	For  $\alpha>1$, $\bond^*$ is a fixed point of $T$ if and only if $x^*=\frac{1}{\bond^*}$ is a fixed point of $\widetilde{T}$. 
\end{lemma}

\cref{alg:1} outlines the procedure to determine the value for $\bond^*$, the fixed point of $T(\bond)$, using the map $\widetilde{T}(x)$ and the equivalence of \cref{lem:FixedPoints}. To evaluate $\widehat{S}_r\xi_m(1)$, we use the polynomial approximation derived in \cref{sec:Legendre3d} to write 
\begin{equation}\label{eq:Sr1}\small
	\widehat{S}_r \xi_m(1) = \sum_{j=1}^n a_j^m \left((\widehat{S}_r h_j^m)(1) + \beta_j^m (\widehat{S}_r h_{j+1}^m)(1) \right),
\end{equation}
where $\widehat{S}_r h_j^m(1) = \frac{\mu_j^m}{2\pi\left(j+m-\frac{1}{2}\right)\left(j-\frac{1}{2}\right)}$, see \cref{sec:AppIOat1} for details. 
\begin{algorithm}[tbhp]
	\caption{Procedure to determine $\bond^*$}\label{alg:1}
	\begin{algorithmic}
		\REQUIRE{$\bond>1$, $\alpha>1$, $m\in \mathbb{N}$}
		\STATE $x_0 \gets 1/\bond$
		\WHILE{$|x_{n+1}-x_n|>$threshold}
		\STATE Solve \cref{eq:IFP_GEP}  for $(\omega,\xi_m)$ with $\bond = 1/x_n$
		\STATE Evaluate $\omega^2(\widehat{S}_r\xi_m)(1)$ using \cref{eq:Sr1}
		\STATE $x_{n} \gets \widetilde{T}(x_n)$
		\ENDWHILE
		\RETURN{$1/x_n$}
	\end{algorithmic}
\end{algorithm}

In \cref{tab:IterativeSteps}, we report on the number of steps for \cref{alg:1} to converge with a threshold of $10^{-14}$ for $\alpha=\{2,3,5,10\}$ and $m=\{1,2,3,4,5\}$. As the number of steps were similar for all the $n$ tested, we only give the results for $n=20$. We note that the classical fixed point iteration corresponds to $\alpha=0$ and the edited fixed point iteration should converge for any $\alpha>1$, although it is expected that more iterations are necessary for convergence for larger $\alpha$.

\begin{table}[tbhp]
	\centering
	\begin{tabular}{|c|c|c|c|c|c|} \hline
				     & $m=1$ & $m=2$ & $m=3$ & $m=4$ & $m=5$ \\ \hline
		$\alpha = 2$  & 29  & 40  & 92  & 195  & 540  \\ \hline
		$\alpha = 3$  & 30  & 94  & 194 & 395  & 1066 \\ \hline
		$\alpha = 5$  & 76  & 198 & 392 & 783  & 2088 \\ \hline
		$\alpha = 10$ & 184 & 450 & 873 & 1726 & 4568 \\ \hline
    	\end{tabular}
	\caption{Number of steps for \cref{alg:1} to converge (with $\text{threshold} = 10^{-14}$) to a fixed point for $\alpha=2,3,5,10$, $m=1,2,3,4,5$ and $n=20$. All other values of $n$ tested produced similar results.}
	\label{tab:IterativeSteps}
\end{table}

Values of $\bond^*$ against the dimension of the solution space, $n=\{5,20,80\}$, and the radial mode, $m=\{1,2,3,4,5\}$ are given in \cref{tab:MagicBo}. We note the agreement between the column $m=1$ and the value given in \cref{fig:IFP_HighSpot_c}. 
\begin{table}[tbhp]
	\centering
	\begin{tabular}{|c|c|c|c|c|c|} \hline
			    & $m=1$ & $m=2$ & $m=3$ & $m=4$ & $m=5$ \\ \hline
		$n=5$  & 4.6342188 & 7.1574495 & 7.8108948 & 6.6284730 & 3.6110536 \\ \hline
		$n=20$ & 4.6346165 & 7.1588900 & 7.8137285 & 6.6328709 & 3.6171135 \\ \hline
		$n=80$ & 4.6346167 & 7.1588910 & 7.8137311 & 6.6328760 & 3.6171218 \\ \hline
	\end{tabular}
	\caption{Values of $\bond^*$ for $n\in \{5,20,80\}$ and radial mode $m\in \{1,2,3,4,5\}$.}
	\label{tab:MagicBo}
\end{table}

\begin{theorem}\label{thm:HSInt}
	Consider the fundamental sloshing frequency $j=1$. If $m\in\{1,\ldots,5\}$, then $\bond^*$ exists. Furthermore,
	\begin{enumerate}
		\item[(i)] If $\bond\in(\bond^*,\infty)$, then the high spot is located in the interior.
		\item[(ii)] If $\bond\in[1,\bond^*)$, then the high spot is on the boundary.
	\end{enumerate}
\end{theorem}

\begin{proof}
	For $m\geq 6$, the equation $T(\bond)=\bond$ has no solution, so the high spot is inside the domain. We note that since $\xi_m(0)=0$ and $\xi_m(r)\geq 0$, $r=0$ is always a global minimum.

	(i) Recall that $x=\frac{1}{\bond}$. From \cref{fig:T_Bo_c}, we observe that $\widetilde{T} (x) > x$ when $x<x^*$ so that $\widetilde{T}\left(\frac{1}{\bond} \right) > \frac{1}{\bond}$ whenever $\bond > \bond^*$. From the definition of $\widetilde{T}$ we have $\frac{1/T(\bond)-\alpha/\bond}{1-\alpha} > \frac{1}{\bond}$ and since $\alpha >1$ this simplifies to $\frac{1}{T(\bond)}<\frac{1}{\bond}$. Therefore, $\frac{\omega^2 \widehat{S}_r \xi_m(1) - 1}{m^2}<\frac{1}{\bond}$ so that $\xi_m''(1) = m^2 + \bond(1-\omega^2 \left(\widehat{S}_r \xi_m)(1) \right)>0.$ Thus $\xi_m''(1)>0$ and along with the Neumann boundary conditions this implies that $\xi_m(1)$ is a local minimum and the high spot is in the interior.

	(ii) Using similar calculations as above, we establish that $\xi_m''(1)<0$. Since $\xi_m'(1)=0$, we have that a local maximum occurs at $r=1$. To conclude that $r=1$ is a global maximum of $\xi_m(r)$, we see by inspection that there are no other critical points in $(0,1)$. Thus, the high spot is on the boundary.  
\end{proof}
\begin{remark}
	When $\bond \to \infty$, \cref{thm:HSInt} implies that the high spot is in the interior for the problem without surface tension. This is consistent with previous results on the 2D IFP \cite[Proposition 2.7]{kulczycki:2009} as well the 3D axially symmetric IFP \cite[Theorem 3.2]{kulczycki:2009} where the authors neglect surface tension.
\end{remark}
For further illustration, \cref{fig:zeros} shows the location of the first two nonzero roots of $\xi_1(r)$ for all $\bond$, one them always being $r=1$. The derivative is expressed analytically as $\xi_1'(r)=\sum_{j=1}^na_j^1\frac{d}{dr}q_j^1(r)$ where the radial polynomial derivative uses derivatives of Jacobi polynomials and a simple root solver is used to find the zeros. \cref{fig:zeros} shows the expected horizontal asymptote $y=0.650312$, the expected horizontal line $r=1$ and the expected critical $\bond^*$ where both curves intersect. For $\bond<\bond^*$, since the blue curve is always above the orange line, see \cref{fig:zeros_b} for a blow up plot, we conclude that there are no critical points in $(0,1)$.
\begin{figure}[t]
	\label{fig:zeros}
	\centering
		\subfloat[{$\bond\in[0.01,1000]$}]{\label{fig:zeros_a}\includegraphics[width=0.49\textwidth,trim = {0cm 0cm 1cm 1.8cm},clip]{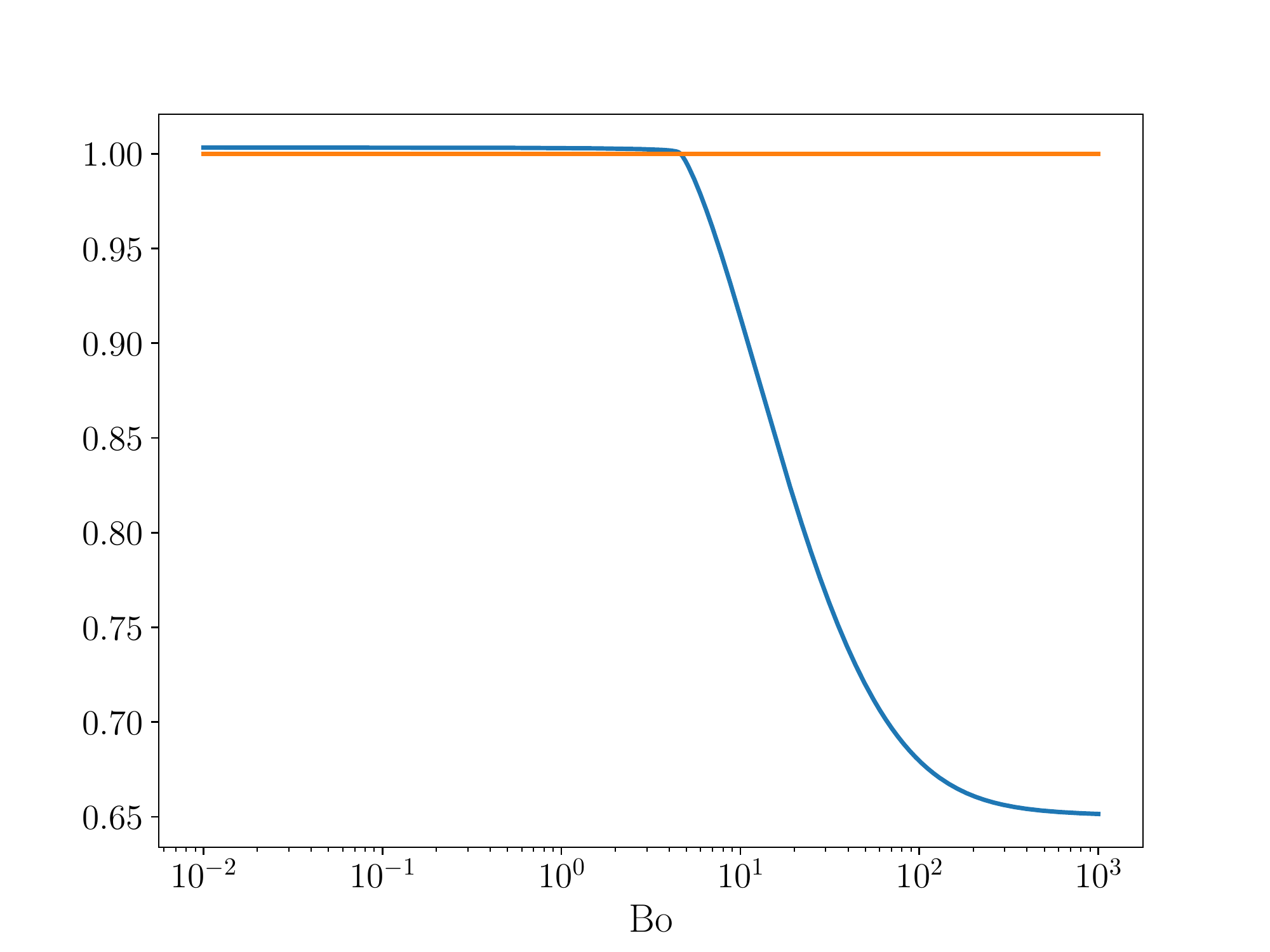}}
		\subfloat[{Blow up for $\bond\in[0.1,5]$}]{\label{fig:zeros_b}\includegraphics[width=0.49\textwidth,trim = {0cm 0cm 1cm 1.8cm},clip]{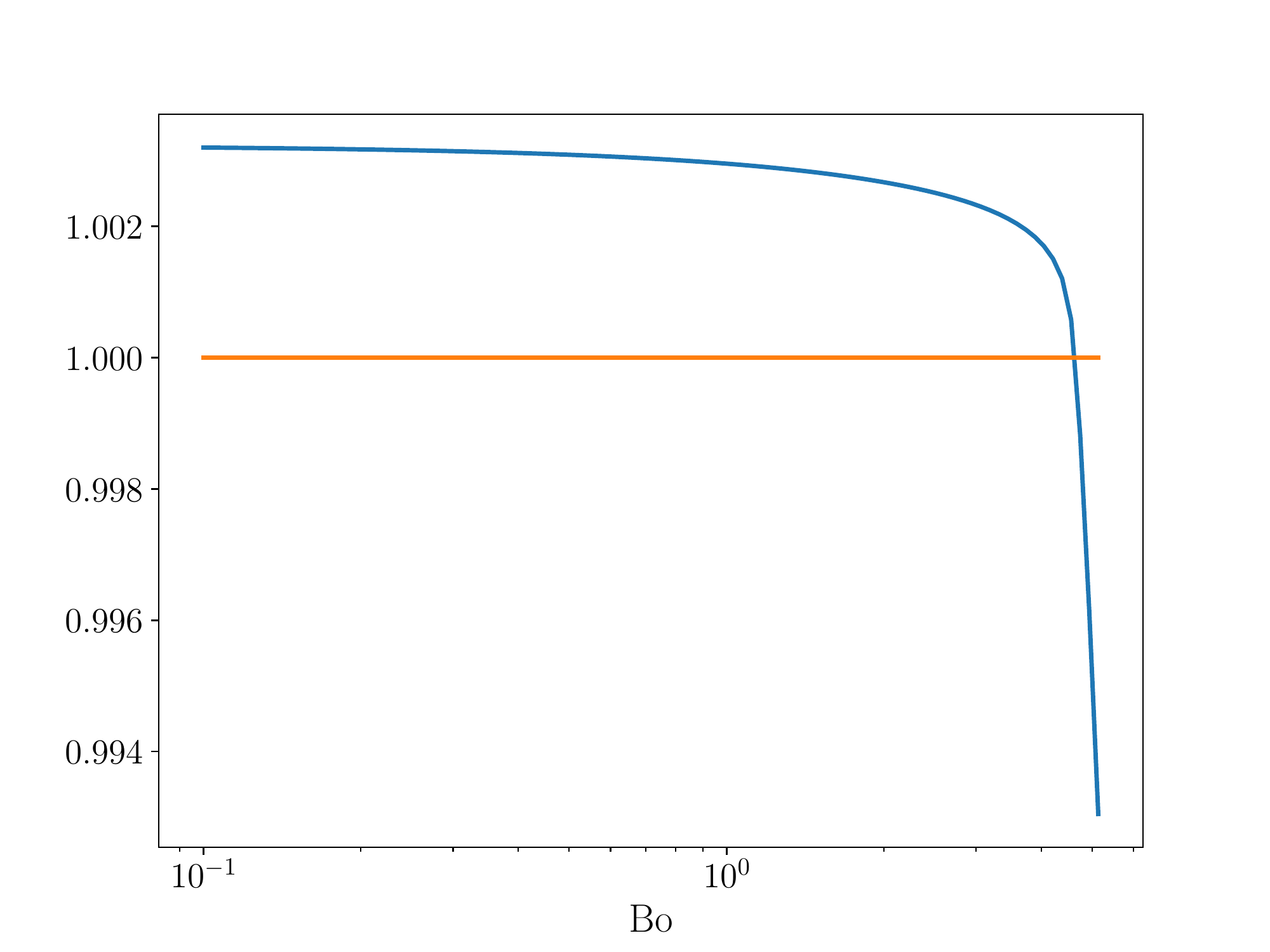}}
	\caption{Location of the first two zeros of $\xi'_m(r)$ as a function of $\bond$. Because of the Neumann boundary condition, $r=1$ is always a zero and corresponds to the orange line. The blue curve is obtained by numerically finding the first zero of $\xi_m'(r)$.}
\end{figure}


\section{Discussion} \label{sec:IFP_disc} 
This paper provides the first study on the ice fishing problem (IFP) including surface tension effects. Although the geometries considered are infinite parallel strips and circular holes, for any $\bond$, the results give an upper bound for the fundamental sloshing frequency for any container with the same free surface \cite{tan:2017}. Building on previous studies without surface tension \cite{davis:1970, henrici:1970}, we derived an integro-differential equation and proved the spectrum is equivalent to the spectrum of \cref{eq:IFP}. The novelty of this approach is that, for the considered geometries, the problem is transformed from an unbounded domain to a bounded one-dimensional domain. We numerically solved it by expanding in a polynomial basis, suitably chosen to satisfy the boundary conditions. We derived a closed form expression for the eigenvalue problem, which is a generalized eigenvalue matrix equation, and numerically approximated it. 

We have  assumed the free-end edge constraint, \ie the contact line slips freely  along the vertical wall (edge of the ice hole) while intersecting it orthogonally. However, experimental evidence \cite{cocciaro:1993, dussan:1979} reveals that the dynamic behavior of the contact line depends crucially on the contact angle, \ie the angle where the free surface meets a solid surface. It would be interesting to  consider either the pinned-end edge constraint \cite{benjamin:1979} or a dynamic contact line boundary conditions such as Hocking's linear wetting boundary condition \cite{hocking:1987a, hocking:1987b}, Dussan's nonlinear contact line model incorporating static contact angle hysteresis \cite{dussan:1979}, or a combination of Hocking's and Dussan's model that was recently proposed by Viola et al. \cite{viola:2018, viola:2019}. 

The eigenvalues from IFP are upper bounds for the corresponding eigenvalues for all other containers with coinciding free surface. It would therefore be of interest to know if the approximate eigenvalues presented here are perhaps upper or lower bounds for the true eigenvalues of IFP. In the absence of surface tension this was indeed considered in \cite{fox:1983,henrici:1970}. The methods for the upper bounds were summarized in \cref{sec:Intro}. For the lower bounds, Henrici et al. \cite{henrici:1970} utilize the domain monotonicity property along with an infinitely deep cylinder or trough to bound the eigenvalues from below. When appropriate, they improve the bound by implementing the Krylov-Bogoliubov inequality on the integral operator to provide a tighter lower bound. Fox et al. \cite{fox:1983} use the method of intermediate problems on an equivalent weighted problem on a region where solutions are known. Taking advantage of known eigenvectors of the unweighted sloshing problem on the simple region, Fox et al. \cite{fox:1983} produce much tighter lower bounds than otherwise observed while using low-dimensional matrices. While the domain monotonicity property holds when including surface tension, the other properties mentioned do not necessarily extend to the model considered here. Nonetheless, numerically bounding eigenvalues would be a significant advantage to this study and could be investigated in future work. 

We numerically study how the location of the high spot depends on $\bond$. Using a simple fixed-point iteration we determine the value of $\bond$ such that for any lower $\bond$ the high spot is on the boundary of the free surface. This result compliments previous work \cite{kulczycki:2009} which shows that the high spot is always on the interior for the IFP in the absence of surface tension. For physical intuition, we recall the fundamental sloshing profile for the circular hole IFP minimizes the free surface energy,
\begin{equation*}
\frac{\int_0^1 \xi_m^2 r \, dr + \frac{1}{\bond} \int_0^1 \left(\frac{m^2}{r} \xi_m^2 + \left(\xi_m'\right)^2 r \right)\, dr}{\int_0^1 (\widehat{S}_r\xi_m)(r) \xi_m(r) r dr}.
\end{equation*}
For simplicity, we scale the solution such that $\int_0^1 (\widehat{S}_r\xi_m)(r) \xi_m(r) r dr=1$. For large $\bond$ we are primarily minimizing $\int_0^1 \xi_m^2 r \, dr$, whereas for small $\bond$ we are primarily minimizing $\int_0^1 \left(\frac{1}{r} \xi_m^2 + \left(\xi_m'\right)^2 r \right)\, dr$. These different contributions to the free surface energy are shown in \cref{fig:Energy}. 
 \begin{figure}[t]
	\centering
		\subfloat[{$\int_0^1 \xi^2 r \, dr$ for $\bond \in [10^0,10^7]$}]{\includegraphics[width=0.48\textwidth,trim = {0.5cm 0.8cm 0.65cm 0.6cm},clip]{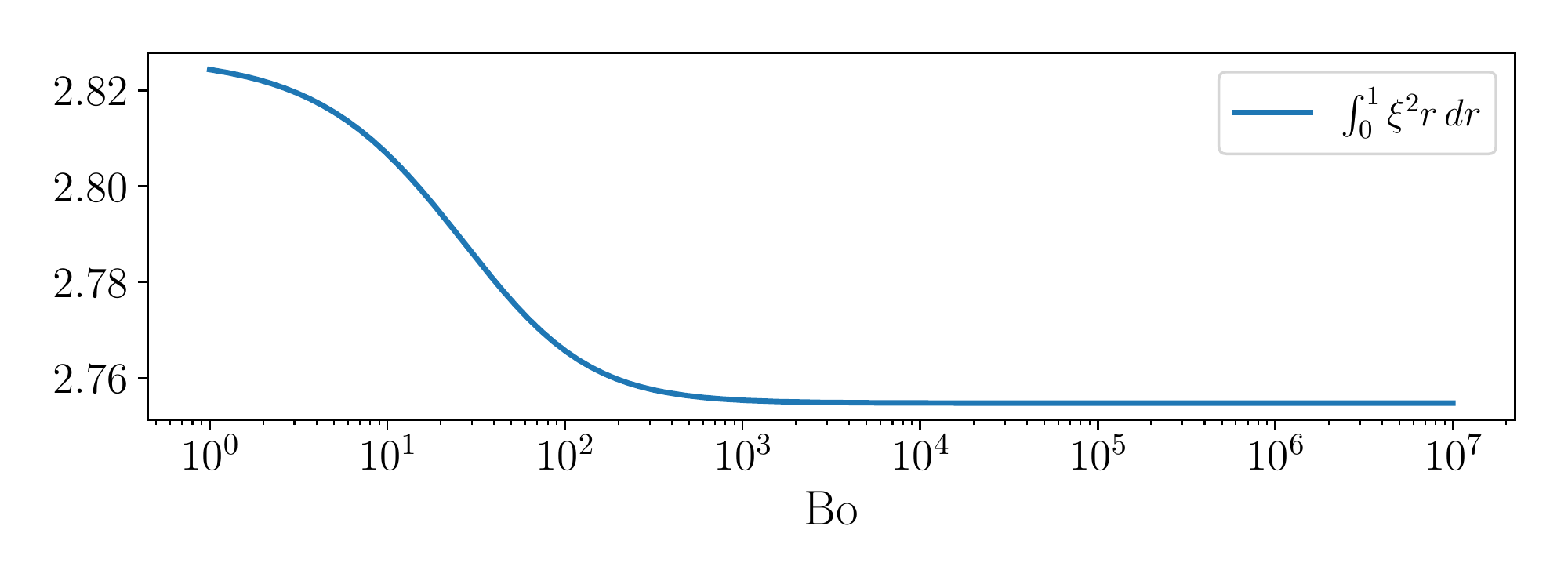}}
		\subfloat[{$\int_0^1 \left(\frac{1}{r} \xi^2 + \left(\xi'\right)^2 r \right)\, dr$ for $\bond \in [10^0,10^7]$}]{\includegraphics[width=0.48\textwidth,trim = {0.7cm 0.8cm 0.65cm 0.6cm},clip]{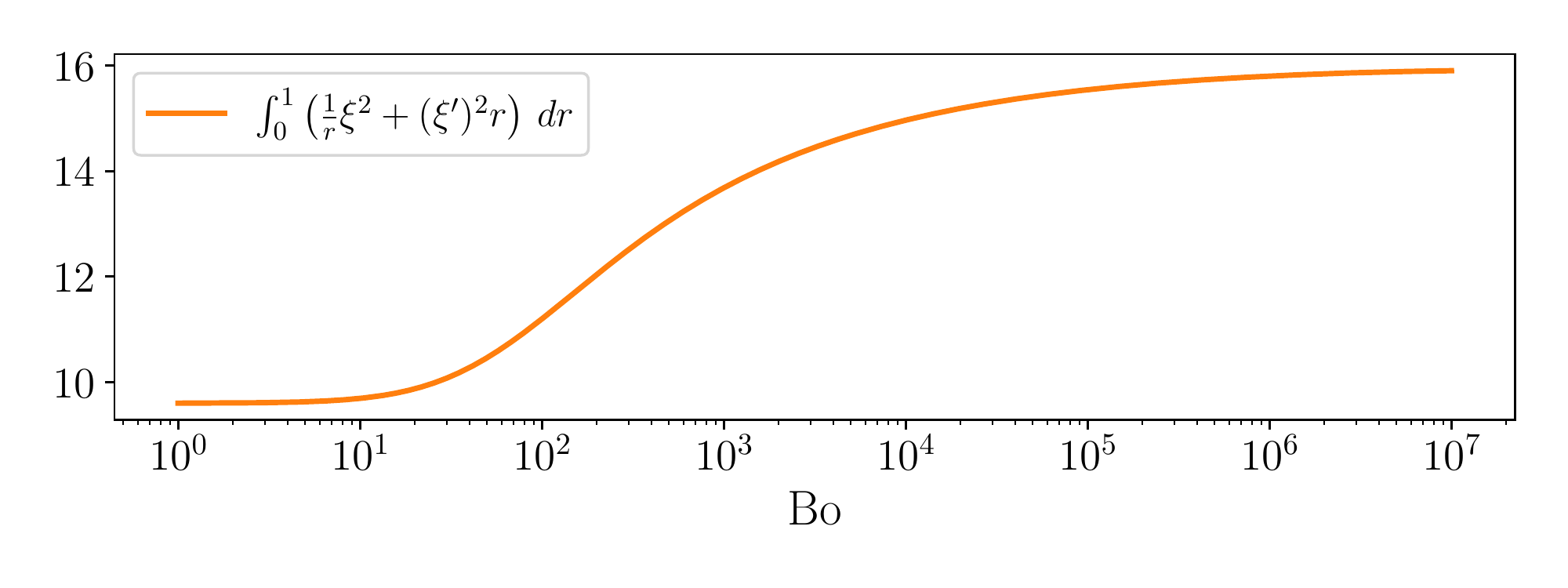}}
	\caption{Comparison of the different contributions to the free surface energy to illustrate how the fundamental solution minimizes one versus the other depending on $\bond$ with $m=1$.}
	\label{fig:Energy}
\end{figure}
Therefore, we pay a more significant penalty for variation in the derivative when $\bond$ is small and the intuition is that sloshing profiles with no interior extrema suffer less from this penalty. Motivated by this observation, we conjecture that \emph{for any shape domain, given sufficiently small $\bond$, the high spot is located on the boundary of the free surface}.


\section*{Acknowledgments} We would like to thank Akil Narayan and Fernando Guevara Vasquez for helpful conversations. 


\clearpage

\appendix


\section{Matrix Elements Computations} \label{sec:App3dMatrices}
The coefficients $M_{ij}^m$ follow immediately from the orthogonality of the $h_j^m$ polynomials and computation for $\tilde{L}_{ij}^m$ can be found in \cite{miles:1972}. Therefore, here we focus on showing that the stiffness matrix is diagonal and finding the values for those coefficients.

We first define $\tilde{K}_{ij}^m=\int_0^1 \left(\frac{m^2}{r}h_i^m h_j^m + (h_i^m)'(h_j^m)' r\right)\, dr $ so that we have $K_{ij}^m = \tilde{K}_{ij}^m+ \beta_i \tilde{K}_{i+1,j}^m + \beta_j \tilde{K}_{i,j+1}^m + \beta_i \beta_j \tilde{K}_{i+1,j+1}^m$. Now we compute the $\tilde{K}_{ij}^m$ elements and start by recalling that $h_j^m = \mu_j^m r ^m P_{j-1}^{(0,m)}(2r^2-1)$. From \cite[18.9.15]{NIST} we have that 
$$(h_j^m)' = \frac{m}{r}h_j^m + 2 \mu_j^m r^{m+1} (m+j) P_{j-2}^{(1,m+1)}(2r^2-1),$$
and substituting this along with the definition of $h_j^m$ into $\tilde{K}_{ij}$ we obtain
\begin{equation*}
	\tilde{K}_{ij}^m = 2\mu_j^m \mu_i^m \left(m^2 A_{ij}^m +\frac{m}{\mu_j^m} (i+m) B_{ij}^m + \frac{m}{\mu_i^m}(j+m)B_{ji}^m + (j+m)(i+m) C_{ij}^m\right),
\end{equation*}
where
\begin{align*}
	A_{ij}^m &= \int_0^1 r^{m-1} P_{j-1}^{(0,m)}(2r^2-1) r^{m-1} P_{i-1}^{(0,m)}(2r^2-1) r\, dr, \\
	B_{ij}^m &= \int_{0}^1 h_j^m r^m P_{i-2}^{(1,m+1)}(2r^2-1)\, dr,  \\
	C_{ij}^m &= 2\int_{0}^1 r^{m+1} P_{j-2}^{(1,m+1)} (2r^2-1) r^{m+1} P_{i-2}^{(1,m+1)}(2r^2-1)\, dr.
\end{align*}
To simplify $A_{ij}^m$ we relate the $P_{j-1}^{(0,m)}$ Jacobi polynomials to the $P_{j-1}^{(0,m-1)}$ Jacobi polynomials in order to use the orthogonality of the $h_j^{m-1}$ polynomials. First use the symmetry relation \cite[Table 18.6.1]{NIST} and then the connection sum formula \cite[18.18.14]{NIST}, with $\gamma = m$, $\alpha = m-1$ , and $\beta = 0$, to get
\begin{equation*} 
	P_{j-1}^{(0,m)}(x) = (-1)^{j} \frac{(j-1)!}{(m+1)_{j-1}} \sum_{\ell=1}^{j} \frac{m+2(\ell-1)}{m} \frac{(m)_{\ell-1}}{(\ell-1)!} (-1)^{\ell}P_{\ell-1}^{(0,m-1)}(x).
\end{equation*}
Since $m+2(\ell-1) = \frac{1}{2}(\mu_{\ell}^{m-1})^2$ we have
\begin{equation} \label{eq:ConSum1}
	r^{m-1} P_{j-1}^{(0,m)}(2r^2-1) = \frac{(-1)^{j}}{m} \frac{(j-1)!}{(m+1)_{j-1}} \sum_{\ell=1}^{j} \frac{\mu_{\ell}^{m-1}}{2}\frac{(m)_{\ell-1}}{(\ell-1)!} (-1)^{\ell}h_{\ell}^{m-1}.
\end{equation}
Let $\kappa_{ij}^m = (-1)^{i+j} \frac{(j-1)!(i-1)!}{(m+1)_{j-1}(m+1)_{i-1}}$ and $\nu = \min(i,j)$, then we have
\begin{equation*}
	A_{ij}^m = \frac{\kappa_{ij}^m}{m^2} \sum_{\ell=1}^{j} \sum_{k=1}^{i} \frac{\mu_{\ell}^{m-1} \mu_k^{m-1}}{4\kappa_{\ell k}^{m-1}} \delta_{\ell k}
	= \frac{\kappa_{ij}^m}{m^2} \sum_{\ell=0}^{\nu-1}(m+2 \ell)\left(\frac{(m)_\ell}{\ell!} \right)^2
	= \frac{\kappa_{ij}^m}{m} \left(\frac{(m+1)_{\nu-1}}{(\nu-1)!}\right)^2.
\end{equation*}
The last simplification follows from $\ds \left(\frac{(\nu-1)!}{(m+1)_{\nu-1}}\right)^2 \sum_{\ell=0}^{\nu-1}(m+2 \ell)\left(\frac{(m)_\ell}{\ell!} \right)^2 = m$, which can be easily proved by induction om $\nu$.
Therefore, 
\begin{equation*}
	A_{ij}^m = \frac{(-1)^{i+j}}{2m} \frac{(\min(i,j))_m}{(\max(i,j))_m}.
\end{equation*}
Next, we turn to $B_{ij}^m$ and start by again using the connection sum formula \cite[18.18.14]{NIST}, this time with $\gamma = 1$, $\alpha = 0$ , and $\beta = m$, to note that
\begin{equation} \label{eq:ConSum2}
	r^m P_{i-1}^{(1,m)}(2r^2-1) = \frac{1}{2(m+i)} \sum_{\ell = 1}^{i} \mu_{\ell}^m h^m_{\ell}(x).
\end{equation}
Therefore, using the contiguous relation \cite[18.9.3]{NIST} for $P_{i-2}^{(1,m+1)}$ along with the above relation for $r^m P_{i-1}^{(1,m)}$ and \cref{eq:ConSum1} for $r^{m} P_{j-1}^{(0,m+1)}$ we have
\begin{equation*}
r^m P_{i-2}^{(1,m+1)}(2r^2-1) = \frac{1}{2} \sum_{\ell=1}^{i} \frac{1}{m+i} \left(1 - (-1)^{i+\ell}  \frac{(\ell)_m}{(i)_m} \right) \mu_{\ell}^m h_{\ell}^m(r). 
\end{equation*}
From the orthogonality of $h_j^m$ we now have
\begin{equation*}
	B_{ij}^m = 
	\begin{dcases}
		\frac{\mu_j^m}{2(m+i)}\left(1 - (-1)^{i+j} \frac{(j)_m}{(i)_m}\right) & j < i, \\
		0 & j \geq i.
	\end{dcases}
\end{equation*}
Lastly, we turn to $C_{ij}^m$ and use \cref{eq:ConSum2} for $r^{m+1}P_{j-2}^{(1,m+1)}$ so that we can again use the orthogonality of the $h_j^m$ polynomials. After simplifying we get  
\begin{equation}
	C_{ij}^m  = \frac{\min(i,j)-1}{\max(i,j)+m}
\end{equation}
With these expressions for $A_{ij}^m$, $B_{ij}^m$, and $C_{ij}^m$ we now have that
\begin{equation*}
	\tilde{K}_{ij}^m = \mu_j^m \mu_i^m \left(m+2(\min(i,j)-1)(\min(i,j)+m) \right).
\end{equation*} 
Let us write the stiffness matrix as $K_{ij}^m = \Omega_{ij}^m + \beta_i \Omega_{i+1,j}^m$ with $\Omega_{ij}^m = \tilde{K}_{ij}^m+\beta_j \tilde{K}_{i,j+1}^m$. Due to the symmetry of $K$ we assume without loss of generality that $j<i$ such that 
$$\tilde{K}_{ij}^m = \mu_j^m \mu_i^m \left(m+2(j-1)(j+m) \right).$$ 
Then, from the definition of $\beta_{j}$ we have that $\Omega_{ij}^m=0$ and $\Omega_{i,j+1}^m=0$ if $j\neq i$.
In the case that $i = j$ it follows from direct calculation that $K_{jj}^m = - \mu_j^m\mu_{j+1}^m\left(\mu_j^{m+1}\right)^2 \beta_j$.


\section{Integral operator evaluated at \texorpdfstring{${r=1}$}{TEXT}}\label{sec:AppIOat1}
We start with the Hankel transform of $h_j^m$, given in \cite{miles:1972}, to determine that 
\begin{equation*}
	\widehat{S} h_j^m(1)= (-1)^{j-1} \mu_j^m \int_0^\infty \frac{J_{2j+m-1}(k)J_m(k)}{k}\, dk.
\end{equation*}
Using the Weber-Schafheitlin formula \cite[10.22.57]{NIST}, with $\alpha = 1 = \gamma$, $\mu = 2j+m-1$, and $\nu = m$, we have
\begin{align*}
	\int_0^\infty \frac{J_{2j+m-1}(k)J_m(k)}{k}\, dk
	&=\frac{\Gamma\left(j+m-\frac{1}{2}\right)}{2\Gamma\left(\frac{3}{2}-j\right)\Gamma\left(j+\frac{1}{2}\right)\Gamma\left(j+m+\frac{1}{2}\right)}, \\
	&=\frac{(-1)^j}{2\pi\left(j+m-\frac{1}{2}\right)\left(j-\frac{1}{2}\right)},
\end{align*}
where we have simplified in a similar fashion as done in \cite{miles:1972}. It follows that
\begin{equation*}
	\widehat{S} h_j^m(1) = \frac{\mu_j^m}{2\pi\left(j+m-\frac{1}{2}\right)\left(\frac{1}{2}-j\right)}.
\end{equation*}

\printbibliography


\end{document}